\documentclass[a4paper,english,oneside]{article}
\usepackage[T1]{fontenc}
\usepackage[utf8]{inputenc}

\usepackage{geometry}
\usepackage{amsmath}
\usepackage{amsfonts}
\usepackage{amssymb}
\usepackage{amsthm}
\usepackage[english]{babel}
\usepackage{color}
\usepackage{lmodern}
\usepackage{hyperref}

\newtheorem{thm}{Theorem}[section]

    \newtheorem{dfn}[thm]{Definition}
    \newtheorem{prop}[thm]{Proposition}
    \newtheorem{lemma}[thm]{Lemma}

    \newtheorem{rmk}[thm]{Remark}

    \DeclareMathOperator\Fuj{Fuj}
    \DeclareMathOperator\crit{crit}

\newcommand{\vertiii}[1]{{\left\vert\kern-0.25ex\left\vert\kern-0.25ex\left\vert #1 
    \right\vert\kern-0.25ex\right\vert\kern-0.25ex\right\vert}}

\title{A global existence result for a semilinear wave equation with scale-invariant damping and mass in even space dimension}

\author{Alessandro Palmieri}

\begin{document}

\maketitle

\begin{abstract} In the present article a semilinear wave equation with scale-invariant damping and mass is considered. The global (in time) existence of radial symmetric solutions in even spatial dimension $n$ is proved using weighted $L^\infty-L^\infty$ estimates, under the assumption that the multiplicative constants, which appear in the coefficients of damping and of mass terms, fulfill an interplay condition which yields somehow a ``wave-like'' model. In particular, combining this existence result with a recently proved blow-up result, a suitable shift of Strauss exponent is proved to be the critical exponent for the considered model. Moreover, the still open part of the conjecture done by D'Abbicco - Lucente - Reissig in  \cite{DabbLucRei15} is proved to be true in the massless case.

\end{abstract}

\section{Introduction}
In the last decade several papers have been devoted to the study of the semilinear wave equation with scale-invariant damping and power nonlinearity
\begin{align} \label{semilinear scale inv damping}
\begin{cases}
u_{tt}- \Delta u +\frac{\mu}{1+t}u_t=|u|^p  , & x\in \mathbb{R}^n, \, t>0, \\
u(0,x)=u_0(x),  & x\in \mathbb{R}^n,  \\
u_t(0,x)=u_1(x), & x\in \mathbb{R}^n,
\end{cases}
\end{align} where $\mu$ is a positive constant. The damping term in \eqref{semilinear scale inv damping} is critical, indeed, it represents a threshold between effective and non-effective dissipation (see \cite{Wirth06, Wirth07}). Here, roughly speaking, effective (non-effective, respectively) means that the solution behaves somehow as the solution of the classical damped wave equation (the free wave equation, respectively) from the point of view of decay estimates. Due to the limit behavior of the time dependent coefficient in the damping term, it is quite natural that the magnitude of the constant $\mu$ influences strongly the nature of the equation.

Naively speaking, we can say that for suitably large $\mu$ \eqref{semilinear scale inv damping} and its corresponding linear Cauchy problem are  ``parabolic'' from the point of view of the critical exponent for the power-nonlinearity and decay estimates, respectively.

More precisely,  global (in time) existence results are proved in \cite{D13} for super-Fujita exponents, that is, for $p>p_{\Fuj}(n)\doteq 1+\tfrac{2}{n}$  in dimensions $n=1,2$ and $n\geq 3$ in the cases $\mu \geq \frac{5}{4}$, $\mu \geq 3$ and $\mu\geq n+2$, respectively. Combining these existence results with a blow-up result from \cite{DabbLucMod}, it results that the critical exponent for \eqref{semilinear scale inv damping} is the so-called \emph{Fujita exponent} $p_{\Fuj}(n)$ when $\mu$ is sufficiently large. 

Simoultaneously and independently, in \cite{Waka14A} with different techniques $p_{\Fuj}(n)$ is proved to be critical, assumed that $\mu$ is greater than a given constant $\mu_0\approx (p-p_{\Fuj}(n))^{-2}$. In particular, the test function method is employed to prove the blow-up of the solution for $1<p\leq p_{\Fuj}(n)$ when $\mu\geq 1$ and for $1<p\leq p_{\Fuj}(n+\mu-1)$ when $\mu\in (0, 1)$, for suitable initial  data. 

Hence, for $\mu$ suitably large, in \eqref{semilinear scale inv damping} the damping term has the same influence on properties of solutions as in the constant coefficients case (classical damped wave equation).

Yet, the situation is completely different when $\mu$ is small. In \cite{DabbLucRei15} the special value $\mu=2$ is considered. Indeed, for this value of $\mu$, \eqref{semilinear scale inv damping} can be transformed in a semilinear free wave equation with nonlinearity $(1+t)^{-(p-1)}|u|^p$. Thus, by using the so-called \emph{Kato's lemma} (see for example \cite[Lemma 2.1]{Yag05} or \cite[Lemma 2.1]{Tak15}), the authors prove a blow-up result for 
\begin{align} \label{def p2} 
1<p\leq p_2(n)\doteq \max\{p_{\Fuj}(n),p_0(n+2)\}=\begin{cases} p_{\Fuj}(n) & \mbox{if} \,\, n=1, \\ p_0(n+2)  & \mbox{if} \,\, n\geq 2, \end{cases}
\end{align} in any spatial dimension, assuming nonnegative and compactly supported initial data; here $p_0(n)$ denotes the so-called \emph{Strauss exponent}, that is, the critical exponent for the free wave equation with power nonlinearity, which is the positive root of the quadratic equation $$(n-1)p^2-(n+1)p-2=0.$$
 
Furthermore, the previous upper bound for $p$ is optimal in the cases $n=1,2,3$, since global existence results are prove for $p>p_2(n)$ in \cite{D13} for $n=1$ and in \cite{DabbLucRei15} for $n=2$ and $n=3$ in the  radial symmetric case. Afterwords, in \cite{DabbLuc15} the sharpness of that blow-up result is shown also in odd dimensions $n\geq 5$ for the radial symmetric case. Since the critical exponent is $p_0(n+2)$ for $n=2$ and any $n\geq 3$,  $n$ odd, we remark for the value $\mu=2$ a ``wave-like'' behavior from the point of view of the critical exponent $p$ in \eqref{semilinear scale inv damping}. Moreover, recently, in several works, namely \cite{LaiTakWak17,IkedaSob17,TuLin17sc,TuLin17cc}, it has been studied the blow-up of solutions to \eqref{semilinear scale inv damping} in the case in which the constant $\mu$ is small.
 
 Roughly speaking, in those papers it is derived $p>p_0(n+\mu)$ as a necessary condition for the global (in time) existence of solutions of \eqref{semilinear scale inv damping}, under suitable assumptions on initial data, for $0<\mu <\frac{n^2+n+2}{n+2}$. Furthermore, some upper bound estimates for the life-span of non-global (in time) solutions are proved. This necessary condition points out once again the hyperbolic nature of the model \eqref{semilinear scale inv damping} for small $\mu$.
%
%

The semilinear model \eqref{semilinear scale inv damping} can be generalized, considering a further lower order term, namely a mass term, whose time dependent coefficient is chosen suitably in order to preserve the scale-invariance property of the corresponding linear model. Therefore,  in this work we will focus on the Cauchy problem for semilinear wave equation with scale-invariant damping and mass and power nonlinearity
\begin{align} \label{semilinear scale inv damping and mass}
\begin{cases}
u_{tt}- \Delta u +\frac{\mu}{1+t}u_t+\frac{\nu^2}{(1+t)^2}u=|u|^p  , & x\in \mathbb{R}^n,  \, t>0,\\
u(0,x)=u_0(x), & x\in \mathbb{R}^n, \\
u_t(0,x)=u_1(x), & x\in \mathbb{R}^n,
\end{cases}
\end{align}  where $\mu,\nu$ are nonnegative constants.
Let us define the quantity  $\delta \doteq (\mu-1)^2 -4\nu^2$, which describes the interplay between the damping term $\frac{\mu}{1+t} u_t$ and the mass term $\tfrac{\nu^2}{(1+t)^2}u$.  For further considerations on how the quantity $\delta$ describes the interplay between the damping term $\frac{\mu}{1+t} u_t$ and the mass term $\tfrac{\nu^2}{(1+t)^2}u$ one can see \cite{PalThesis}.

Recently, \eqref{semilinear scale inv damping and mass} has been studied in \cite{NunPalRei16,PalRei17,Pal17,PalRei17FvS,Pal18,DabbPal18}  under different assumptions on $\delta$. 

In this article, the following relation between $\mu$ and $\nu$ is required:
\begin{align}\label{delta=1 condition}
\delta =1.
\end{align}

We stress, that \eqref{delta=1 condition} allows to relate the solution to \eqref{semilinear scale inv damping and mass} with the solution to the semilinear Cauchy problem 
\begin{align} \label{semilinear free wave eq}
\begin{cases}
v_{tt}- \Delta v = (1+t)^{-\frac{\mu}{2}(p-1)}|v|^p  , & x\in \mathbb{R}^n, \, t>0,\\
v(0,x)=u_0(x), & x\in \mathbb{R}^n, \\
v_t(0,x)=u_1(x)+\tfrac{\mu}{2}u_0(x), & x\in \mathbb{R}^n,
\end{cases}
\end{align}
through the transformation $u(t,x)=(1+t)^{-\frac{\mu}{2}}v(t,x)$.

Supposing the validity of \eqref{delta=1 condition}, in \cite{NunPalRei16} has been proved a blow-up result for $$1<p\leq p_{\crit}(n,\mu) \doteq \max\{p_{\Fuj}\big(n+\tfrac{\mu}{2}-1\big),p_0(n+\mu)\}=\begin{cases} p_{\Fuj}\big(n+\tfrac{\mu}{2}-1\big) & \mbox{if} \,\, n=1,2, \\ p_0(n+\mu)  & \mbox{if} \,\, n\geq 3, \end{cases}$$ 
provided that data are nonnegative and compactly supported (see also \cite[Theorem 2.6]{NunPalRei16}).
Very recently, in \cite{Pal18} the exponent $p_{\crit}(n,\mu)$ is shown to be critical in the odd dimensional case $n\geq 1$. In particular, following the approach of \cite{Ku94A,KuKu94,Ku94,KuKu95,Ku97}, in the odd dimensional case $n\geq 3$ the radial symmetric case is considered, but an upper bound for $\mu$ has to be required. 

Since in even dimension Huygens' principle is no longer valid, it is clear that something has to be modified with respect to the approach in \cite{Pal18}, in order to study the even case. Purpose of this work is study the sufficiency part for \eqref{semilinear scale inv damping and mass} under the assumption \eqref{delta=1 condition} when $n\geq 4$ is even. In other words, we want to prove that $p_0(n+\mu)$ is actually the critical exponent, by proving a global (in time) existence result for supercritical exponents. However, due to technical reasons, it is necessary to claim $\mu$ below a certain threshold (exactly as in the odd dimensional case we mentioned above).
More specifically, in the treatment we will follow the approach developed in \cite{KuKu96} for the free wave equation in the radial symmetric and even dimensional case.  


\paragraph{Notations} In the present paper we denote $\langle y\rangle=1+|y|$ for any $y\in\mathbb{R}$. Furthermore, $f\lesssim g$ means $0\leq f\leq C g$ for a suitable, independent of $f$ and $g$ constant $C>0$ and $f\approx g$ stands for $f\lesssim g$ and $f\gtrsim g$. Finally, as in the introduction, throughout the article $p_{\Fuj}(n)$ and $p_0(n)$ denote the Fujita exponent and the Strauss exponent, respectively.

 
 

\section{Main result} \label{Section main result}

In this section we state the global (in time) existence result. But first, let us introduce some preparatory  definitions.
Using the so-called \emph{dissipative transformation} $v(t,x)=\langle t\rangle^{\frac{\mu}{2}} u(t,x)$, thanks to \eqref{delta=1 condition} we find that $v$ is a solution to \eqref{semilinear free wave eq}.
Due to the fact that we are looking for radial solutions, we are interested to solutions of \eqref{semilinear free wave eq} that solve
\begin{align}\label{semi rad n>=5}
\begin{cases}
 v_{tt} -v_{rr}-\frac{n-1}{r}v_r =\langle t\rangle^{-\frac{\mu}{2}(p-1)}|v|^p,&  r>0,\,\, t>0, \\  v(0,r)=f(r), & r>0,\\  v_t(0,r)=g(r), & r>0,
\end{cases}
\end{align}  where  where $f\doteq u_0$ and $g\doteq u_1+\tfrac{\mu}{2}u_0$, $r=|x|$ and a singular behavior of solutions and their $r$-derivatives is allowed as $r\to 0^+$.  Let us recall some known result for the corresponding linear problem
\begin{align}\label{CP radial lin n>=5}
\begin{cases}
 v_{tt}-v_{rr}-\frac{n-1}{r}v_r=0 , & r>0,\,\, t>0,  \\  v(0,r)=f(r), & r>0,\\  v_t(0,r)=g(r), & r>0.
\end{cases}
\end{align}
 Let us begin with the linear Cauchy problem
\begin{align}\label{CP radial lin n>=4, f=0}
\begin{cases}
 v_{tt}-v_{rr}-\frac{n-1}{r}v_r=0 , & r>0,\,\, t>0,  \\  v(0,r)=0, & r>0,\\  v_t(0,r)=c_n g(r), & r>0,
\end{cases}
\end{align} where we included the multiplicative constant $$c_n\doteq  \sqrt{\pi}\, \Gamma\big(\tfrac{n-1}{2}\big)=\pi\, 2^{-\frac{n-2}{2}} (n-3)!!$$ in the second data in order to ``normalize'' the representation formula for the solution of this Cauchy problem.
Let us introduce the parameter $m\doteq \tfrac{n-2}{2}\geq 1.$

We define for $t\geq 0$ and $r>0$ the function
\begin{align}
\Theta(g)(t,r)\doteq r^{-2m}\{w_1(t,r)+w_2(t,r)\}, \label{def theta(g) w1,w2}
\end{align} 
where we have set 
\begin{align}
w_1(t,r) &\doteq  \int_{|t-r|}^{t+r} \lambda^{2m+1} g(\lambda)\, K_m(\lambda,t,r)\,d\lambda, \label{def w1} \\
w_2(t,r) &\doteq \int_{0}^{(t-r)_+} \lambda^{2m+1} g(\lambda)\, \widetilde{K}_m(\lambda,t,r)\,d\lambda, \label{def w2}
\end{align} with
\begin{align}
K_j(\lambda,t,r)&\doteq\int_{\lambda}^{t+r} \frac{H_j(\rho,t,r)}{\sqrt{\rho^2-\lambda^2}} \, d\rho \qquad \mbox{for} \quad j=0,1, \cdots ,m, \label{def Kj} \\
\widetilde{K}_j(\lambda,t,r)&\doteq\int_{t-r}^{t+r} \frac{H_j(\rho,t,r)}{\sqrt{\rho^2-\lambda^2}} \, d\rho \qquad \mbox{for} \quad j=0,1, \cdots ,m, \label{def Ktildej}
\end{align} and
\begin{align*}
H_j(\rho,t,r)&\doteq \big(\big(\tfrac{1}{2 \rho} \tfrac{\partial}{\partial \rho}\big)^*\big)^j H(\rho-t,r)  \qquad \mbox{for} \quad j=0,1, \cdots ,m \quad \mbox{and} \quad |\rho-t|\leq r,   \\
H(\rho,r)&\doteq (r^2-\rho^2)^{m-\frac{1}{2}}, 
\end{align*} being $(\frac{1}{2 \rho} \frac{\partial}{\partial \rho})^*=\frac{\partial}{\partial \rho}(-\frac{1}{2 \rho})$ the adjoint operator of $ \frac{1}{2 \rho} \frac{\partial}{\partial \rho}$ (for further considerations on the representation formula \eqref{def theta(g) w1,w2} cf. \cite[Section 3.2]{KuKu96}).

Let $v^0=v^0(t,r)$ be the function defined as follows: 
 \begin{align} \label{def v0 even case} 
v^0\doteq c_n^{-1} \big\{\Theta(g)+\partial_t \Theta(f)\big\}.
\end{align} The function $v^0$ is the solution to \eqref{CP radial lin n>=5}. In Section \ref{Section lin eq} we will clarify more precisely in which sense $v^0$ solves  \eqref{CP radial lin n>=5} (cf. Proposition \ref{thm lin probl n>=4 distr}). 
We introduce now the space for solutions.
Given a positive parameter  $\kappa$, we define the Banach space
$$X_\kappa\doteq \big\{v\in \mathcal{C}\big([0,\infty),\mathcal{C}^1(0,\infty)\big): \, \|v\|_{X_\kappa <\infty}\big\},$$ equipped with the norm 
\begin{align*}
\|v\|_{X_\kappa}\doteq \sup_{t\geq 0\,, \, r>0}\big(r^{m-1}\langle r\rangle |v(t,r)|+r^{m}|\partial_r v(t,r)|\big) \phi_\kappa (t,r)^{-1},
\end{align*}   where the weight function $\phi_\kappa$ is defined by 
\begin{align}\label{def phi kappa}
\phi_\kappa(t,r)\doteq \langle t+r \rangle^{-\frac{1}{2}} \langle t-r \rangle^{-\kappa}.
\end{align}

Let us consider the integral operator $L$ defined for any $v\in X_\kappa$ by
\begin{align}\label{def Lv case n even}
Lv(t,r)\doteq c_n^{-1} \int_0^t \langle\tau\rangle^{-\frac{\mu}{2}(p-1)}\Theta (|u(\tau,\cdot)|^p)(t-\tau,r) \, d\tau ,
\end{align} where $\Theta(|u(\tau,\cdot)|^p)$ is defined by \eqref{def theta(g) w1,w2}, replacing $g(\lambda)$ with $|u(\tau,\lambda)|^p$. 
 
 According to Duhamel's principle, we introduce the following definition:
\begin{dfn} \label{def sol n even} Let $v^0=v^0(t,r)$ be the function defined through \eqref{def v0 even case}. We say that $v=v(t,r)$ is a radial solution to \eqref{semilinear free wave eq} in $X_\kappa$, if $v\in X_\kappa$ for some $\kappa>0$ and $v$ satisfies the integral equation 
\begin{align*}
v(t,r)=v^0(t,r)+L v(t,r) \qquad \mbox{for any} \quad t\geq 0, \, r>0.
\end{align*}
\end{dfn}
 
In the following we will use the notation:
\begin{align}\label{definition M1(n)}
M(n)\doteq \tfrac{n-1}{2}\Big(1+\sqrt{\tfrac{n+7}{n-1}}\Big).
\end{align} 
As we will see in the next result, which is the main theorem of this article, $M(n)$ is  the upper bound for the coefficient $\mu$.  
 
\begin{thm}\label{thm GER n even semi} Let $n\geq 4$ be an even integer. Let us assume  $\mu\in\big[2,M(n)\big)$ and $\nu\geq 0$ satisfying the relation \eqref{delta=1 condition}, where $M(n)$ is defined by \eqref{definition M1(n)}, and
\begin{align} \label{p assumption GET radial n>4 even}
p\in \Big(p_0(n+\mu),\min\big\{p_{\Fuj}\big(\tfrac{n+\mu-1}{2}\big),p_{\Fuj}(\mu)\big\}\Big). 
\end{align} 

Then, there exist $\varepsilon_0>0$ and $0<\kappa_1<\kappa_2< m+\frac{1}{2}$ such that for any $\varepsilon\in (0,\varepsilon_0)$ and any radial data $u_0\in \mathcal{C}^2(\mathbb{R}^n),u_1\in \mathcal{C}^1(\mathbb{R}^n)$ satisfying
\begin{align*}
|d_r^ju_0(r)| & \leq \varepsilon\, r^{-m-j+1}\langle r\rangle^{-\bar{\kappa}-\frac{3}{2}} \qquad \mbox{for} \quad j=0,1,2, \\
|d_r^j(u_1(r)+\tfrac{\mu}{2}u_0(r))| & \leq \varepsilon \, r^{-m-j}\langle r\rangle^{-\bar{\kappa}-\frac{3}{2}} \qquad \quad \mbox{for}  \quad j=0,1,
\end{align*} for some $\bar{\kappa}\in (\kappa_1,\kappa_2]$, the Cauchy problem \eqref{semilinear scale inv damping and mass} admits a uniquely determined radial solution $u\in \mathcal{C}([0,\infty), \mathcal{C}^1(\mathbb{R}^n\setminus \{0\}))$, in the sense that $v(t,r)=\langle t\rangle^{\frac{\mu}{2}}u(t,r)$ satisfies Definition \ref{def sol n even} for any $\kappa \in (\kappa_1,\bar{\kappa}]$.

Furthermore, the following decay estimates hold for any $t\geq 0$, $r>0$ and $\kappa \in (\kappa_1,\bar{\kappa}]$:
\begin{align*}
|u(t,r)|&\lesssim \varepsilon \,r^{-m+1} \langle r\rangle^{-1} \langle t\rangle^{-\frac{\mu}{2}} \langle t-r \rangle^{-\kappa} \langle t+r\rangle^{-\frac{1}{2}}, \\
|\partial_r u(t,r)|&\lesssim \varepsilon \, r^{-m} \langle t\rangle^{-\frac{\mu}{2}} \langle t-r \rangle^{-\kappa} \langle t+r\rangle^{-\frac{1}{2}}.
\end{align*} 
\end{thm}

For the proof of Theorem \ref{thm GER n even semi} it is necessary to modify some tools developed for the proof of the main theorem in \cite{KuKu96}. 

%
 
The remain part of the article is organized as follows: in Section \ref{Section lin eq} we will recall some known results for the linear problem, following the treatment of \cite{KuKu96}; hence, in Section \ref{Section preliminaries} some preparatory results are derived; also, using these preliminary estimates the proof of Theorem \ref{thm GER n even semi} is provided in Section \ref{Section proof of the main thm}. Finally, in Section \ref{Section final rmk} some final remarks and comments on open problems are given. 
\section{Linear equation} \label{Section lin eq}

In this section we recall some known estimates from \cite[Sections 3-4-5]{KuKu96}, which will be useful for the proof of Theorem \ref{thm GER n even semi}. 
In Section \ref{Section main result}, we introduced the definition of $\Theta(g)$. Now, we will show an alternative representation for $\Theta(g)$ and a representation for its $r$-derivative involving the kernels $K_{m-1}(\lambda,t,r)$ and $\widetilde{K}_{m-1}(\lambda,t,r)$.

\begin{lemma} Let $g$ be a $\mathcal{C}^1((0,\infty))$ function such that $g^{(j)}(\lambda)=O(\lambda^{-2m-j+\varsigma}) $ as $ \lambda\to 0^+$ for $j=0,1$ and some $\varsigma>0$. 
Then, it holds for $t\geq 0$, $r>0$ and $t\neq r$ 
\begin{align}\label{def theta(g) w3,w4}
2r^{2m}\Theta(g)(t,r)\doteq w_3(t,r)+w_4(t,r)
\end{align} with
\begin{align}
w_3(t,r) &\doteq  \int_{|t-r|}^{t+r} \partial_\lambda(\lambda^{2m}g(\lambda))\, K_{m-1}(\lambda,t,r) \, d\lambda , \label{def w3} \\
w_4(t,r) &\doteq  \int_{0}^{t-r} \partial_\lambda(\lambda^{2m}g(\lambda))\, \widetilde{K}_{m-1}(\lambda,t,r) \, d\lambda \qquad \mbox{for} \quad t>r , \label{def w4 t>r} \\
w_4(t,r) &\doteq  (r-t)^{2m}g(r-t) \,K_{m-1}(r-t,t,r) \,\,\qquad \mbox{for} \quad t<r , \label{def w4 t<r} 
\end{align} where $K_{m-1}$ and $\widetilde{K}_{m-1}$ are defined by \eqref{def Kj} and \eqref{def Ktildej}, respectively.

Furthermore, it holds for $t\geq 0$, $r>0$ and $t\neq r$  
\begin{align}\label{def dr theta(g) w5,w6}
\partial_r \big(2r^{2m}\Theta(g)(t,r)\big)\doteq w_5(t,r)+w_6(t,r),
\end{align} where we have set 
\begin{align}
w_5(t,r) &\doteq  \int_{|t-r|}^{t+r} \partial_\lambda(\lambda^{2m}g(\lambda)) \, \partial_r K_{m-1}(\lambda,t,r) \, d\lambda , \label{def w5} \\
w_6(t,r) &\doteq  \int_{0}^{t-r} \partial_\lambda(\lambda^{2m}g(\lambda))\, \partial_r\widetilde{K}_{m-1}(\lambda,t,r) \, d\lambda \qquad \mbox{for} \quad t>r , \label{def w6 t>r} \\
w_6(t,r) &\doteq  (r-t)^{2m}g(r-t)\, \partial_r K_{m-1}(r-t,t,r) \,\,\qquad \mbox{for} \quad t<r . \label{def w6 t<r}
\end{align}
\end{lemma}

\begin{proof}
See \cite[Lemma 4.6]{KuKu96}.
\end{proof}

Using the above described operator $\Theta$, we have seen how to provide through \eqref{def v0 even case} the representation formula for the solution to \eqref{CP radial lin n>=5} in the case $n$ even. 

In the next result, which describes the properties of the solution $v^0$ to \eqref{CP radial lin n>=5}. A condition, which allows the data to be possibly singular as $r\to 0^+$, will be introduced. For the proof of the forthcoming proposition one can see \cite[Theorem 2.1]{KuKu96}.

\begin{prop}\label{thm lin probl n>=4 distr} Let us consider an even integer $n\geq 4$ and radial initial data $f\in \mathcal{C}^2((0,\infty)), \,g\in \mathcal{C}^1((0,\infty))$ such that 
\begin{align}
|f^{(j)}(r)| & \leq \varepsilon\, r^{-m-j+1}\langle r\rangle^{-\kappa-\frac{3}{2}} \qquad \mbox{for} \quad j=0,1,2, \label{cond decay f n even}\\
|g^{(j)}(r)| & \leq \varepsilon \, r^{-m-j}\langle r\rangle^{-\kappa-\frac{3}{2}} \qquad \quad \mbox{for}  \quad j=0,1, \label{cond decay g n even}
\end{align} where the parameters $\varepsilon,\kappa$ satisfy $\varepsilon>0$ and $0<\kappa< m+\frac{1}{2}$.
Let  $v^0=v^0(t,r)$ be defined by \eqref{def v0 even case}. Then, $v^0\in \mathcal{C}^1([0,\infty)\times (\mathbb{R}^n\setminus \{0\}))\cap \mathcal{C}^2([0,\infty), \mathcal{D}'(\mathbb{R}^n))$ is the uniquely determined radial symmetric distributional solution to \eqref{CP radial lin n>=5}, in the following sense:
\begin{align*}
 & \frac{d^2}{dt^2} \langle v^0(t,\cdot),\psi\rangle =\langle v^0(t,\cdot),\Delta \psi\rangle \qquad \mbox{for any} \,\,\, t>0,  \\
 & \langle v^0(t,\cdot) , \psi\rangle \big|_{t=0} = \langle f,\psi\rangle,\quad  \frac{d}{dt}\langle v^0(t,\cdot) , \psi\rangle \big|_{t=0} = \langle g,\psi\rangle,
\end{align*} for any $\psi\in\mathcal{C}^\infty_0(\mathbb{R}^n)$, here $\langle \cdot,\cdot\rangle$ denotes the real scalar product in $L^2(\mathbb{R}^n)$ and $r=|x|$.

Besides, the solution $v^0$ fulfills the following decay estimates for any $t\geq 0$ and $r>0$:
\begin{equation} \label{decay estimate lin n>=4 even delta=1}
\begin{split}
|v^0(t,r)|&\lesssim \varepsilon \, r^{1-m} \langle r\rangle^{-1} \langle t+r\rangle^{-\frac{1}{2}} \langle t-r \rangle^{-\kappa}, \\
|\partial_r v^0(t,r)|&\lesssim \varepsilon \, r^{-m}  \langle t+r\rangle^{-\frac{1}{2}} \langle t-r \rangle^{-\kappa}.
\end{split}
\end{equation}
\end{prop}

\begin{rmk} In the setting of the previous theorem it is possible to derive stronger decay estimates than those we have written in the statement, see \cite[Proposition 4.9]{KuKu96}. Nevertheless, \eqref{decay estimate lin n>=4 even delta=1} is enough for our purposes, in order to prove the global existence result for the semilinear radial symmetric Cauchy problem \eqref{semi rad n>=5} for $n\geq 4$ even. Indeed, it follows by \eqref{decay estimate lin n>=4 even delta=1} that
\begin{align}
\| v^0\|_{X_\kappa}\lesssim \varepsilon  \label{control linear part of the solution n>=4} \qquad \mbox{for} \quad 0<\kappa<m+\tfrac{1}{2}.
\end{align} 
\end{rmk}

Finally, let us recall some known estimates for the kernels $K_j(\lambda,t,r)$ and $\widetilde{K}_j(\lambda,t,r)$, which are going to be used in the treatment of the semilinear case. For the proof of the next lemmas see \cite[Lemma 4.2 and Corollary 4.3]{KuKu96}.

\begin{lemma} \label{Lemma Kj and Ktildej j=m-1,m} Let $\gamma\in \{0,\frac{1}{2}\}$ and let $t\geq 0, r>0$. Let us consider $\alpha=0,1$. Then, we have for $|t-r|<\lambda < t+r$
\begin{align}
|K_m(\lambda,t,r)| &\lesssim r^{m+\gamma-\frac{1}{2}} \lambda^{-m-\gamma} (\lambda-t+r)^{-\frac{1}{2}} , \label{estimate K m} \\
|\partial^\alpha_r K_{m-1}(\lambda,t,r)| &\lesssim r^{m+\gamma+\frac{1}{2}-\alpha} \lambda^{-m-\gamma+1} (\lambda-t+r)^{-\frac{1}{2}} , \label{estimate K m-1} 
\end{align} while for $0<\lambda< t-r$ we get
\begin{align}
|\widetilde{K}_m(\lambda,t,r)| &\lesssim r^{m+\gamma-\frac{1}{2}} (t-r)^{-m-\gamma} (t-r-\lambda)^{-\frac{1}{2}} , \label{estimate Ktilde m} \\
|\partial^\alpha_r \widetilde{K}_{m-1}(\lambda,t,r)| &\lesssim r^{m+\gamma+\frac{1}{2}-\alpha} (t-r)^{-m-\gamma+1} (t-r-\lambda)^{-\frac{1}{2}}. \label{estimate Ktilde m-1} 
\end{align}
\end{lemma}


\begin{lemma} \label{Lemma Ktildej der} Let $\gamma\in \{0,\frac{1}{2}\}$ and let $t\geq 0, r>0$.  Let us consider $j=0,1,\cdots,m$ and $\alpha\in \mathbb{N}$ such that $j+\alpha\leq m$.  If $0<\lambda< t-r$, then, 
\begin{align}\label{estimate Ktilde j der}
|\partial_\lambda \partial^\alpha_r \widetilde{K}_j(\lambda,t,r)|\lesssim r^{2m-j+\gamma-\frac{1}{2}-\alpha}(t-r)^{-j-\gamma}(t-r-\lambda)^{-\frac{3}{2}}.
\end{align}
\end{lemma}


\section{Preliminaries} \label{Section preliminaries}

In this section, we derive some estimates which will be fundamental in the proof of Theorem \ref{thm GER n even semi}, making it more fluent.

Throughout this section we assume the following conditions on $p>1$ and $\kappa$: for the exponent of the nonlinearity  $p$ we require
\begin{align}
 & p_0(n+\mu)<p< p_{\Fuj}\big(\tfrac{n+\mu-1}{2}\big); \label{p>p0(n+mu1) prel n even} 
\end{align}  while for the parameter $\kappa$, which appears in the definition of the norm on $X_\kappa$, we require as upper and lower bounds
\begin{align}
 & 0<\kappa \leq q+\tfrac{\mu}{2}(p-1)= \tfrac{n+\mu-1}{2}p -\tfrac{n+\mu+1}{2}, \label{upper bound for kappa n even} \\
 & k> \tfrac{-q+1}{p-1}-\tfrac{\mu}{2}=\tfrac{2}{p-1}-\tfrac{n+\mu-1}{2} \qquad \Leftrightarrow \qquad p\kappa +q+\tfrac{\mu}{2}(p-1)>\kappa+1, \label{lower bound for kappa n even}  
\end{align} respectively, where 
\begin{align}\label{def q n even}
q\doteq \tfrac{n-1}{2}p-\tfrac{n+1}{2}.
\end{align}

In particular, \eqref{p>p0(n+mu1) prel n even} implies the nonemptiness of the range for $\kappa$, since the upper bound for $p$ provides a positive lower bound for $\kappa$ in \eqref{upper bound for kappa n even}, while the lower bound for $p$ is equivalent to require the validity of the relation
$\tfrac{-q+1}{p-1}-\tfrac{\mu}{2}< q+\tfrac{\mu}{2}(p-1)$, which provides the compatibility between \eqref{upper bound for kappa n even} and \eqref{lower bound for kappa n even}. Besides, the range for $p$ is not empty since $p_0(n+\mu)< p_{\Fuj}\big(\tfrac{n+\mu-1}{2}\big)$ is always true. 

In Section \ref{Section main result} we defined the integral operator $L$. In order to estimate the integrand in \eqref{def Lv case n even}, similarly to \eqref{def theta(g) w1,w2}, \eqref{def theta(g) w3,w4} and \eqref{def dr theta(g) w5,w6}, the following representations are valid for any $ 0\leq \tau \leq t$ and $ r>0$ such that $t\neq r$:
\begin{align}
r^{2m} \Theta(|u(\tau,\cdot)|^p)(t,r) &= W_1(t,r;\tau) +W_2(t,r;\tau), \label{def W1+W2} \\
2r^{2m} \Theta(|u(\tau,\cdot)|^p)(t,r) &= W_3(t,r;\tau) +W_4(t,r;\tau), \label{def W3+W4} \\
\partial_r\big(2r^{2m} \Theta(|u(\tau,\cdot)|^p)(t,r)\big) &= W_5(t,r;\tau) +W_6(t,r;\tau), \label{def W5+W6} 
\end{align} where $W_i(t,r;\tau)$, for $i=1,\cdots,6$, is defined analogously to $w_i(t,r)$ by substituting $|u(\tau,\lambda)|^p$ in place of $g(\lambda)$ in \eqref{def w1}, \eqref{def w2}, \eqref{def w3}, \eqref{def w4 t>r}, \eqref{def w4 t<r}, \eqref{def w5}, \eqref{def w6 t>r} and \eqref{def w6 t<r}.

We introduce now a quantity which prescribes somehow the decay rate  we allow for the nonlinearity $|v|^p$. 
We define for any $v\in X_\kappa$, $j=0,1$ and $\nu\in \mathbb{R}$ the quantity
\begin{align}\label{def Nj nu v^p}
N_j^\nu(|v|^p)\doteq  \sup_{t\geq 0,r>0} \,\big|\partial^j_\lambda \big(\lambda^{2m} |v(\tau,\lambda)|^p\big)\big|\lambda^{-m-\nu+j}\langle \lambda \rangle^{q-\frac{p}{2}+\frac{3}{2}+\nu-j}\phi_\kappa(\tau,\lambda)^{-p},
\end{align} where $q$ and $\phi_\kappa$ are defined by \eqref{def q n even} and \eqref{def phi kappa}, respectively. 



Let us prove now some preliminary lemmas  which are going to be useful in the proof of the last proposition of this section. A fondamental tool for their proofs is the next estimate, which is taken from \cite{KuKu95} (cf. Lemma 4.7).

If $a,b\geq 0$ satisfy $a+b>1$, then, it holds 
\begin{align} \label{KuKu95 lemma}
\int_\mathbb{R} \langle x\rangle ^{-a} \langle x+y \rangle^{-b} \, dx \lesssim 1    \qquad \mbox{for any} \quad y \in \mathbb{R}.
\end{align}

%

\begin{lemma} \label{Lemma 1st integral n even}
Let us consider $p,\kappa$ satisfying \eqref{p>p0(n+mu1) prel n even}, \eqref{upper bound for kappa n even} and \eqref{lower bound for kappa n even} and let $q$ be defined by \eqref{def q n even}. Then, we have for any $y\in \mathbb{R}$
\begin{align*}
\int_\mathbb{R} \langle x\rangle^{-p\kappa} \langle x+y \rangle^{-q-\frac{\mu}{2}(p-1)} \,dx \lesssim \langle y\rangle ^{-\kappa}. 
\end{align*}
\end{lemma}

\begin{proof}
We follow \cite[Lemma 4.8]{KuKu95}. Let us denote $G(y)\doteq \int_\mathbb{R} \langle x\rangle^{-p\kappa} \langle x+y \rangle^{-q-\frac{\mu}{2}(p-1)} \,dx$. We consider first the case $y\geq 0$. We split $G(y)$ as follows:
\begin{align*}
G(y)&= \int _{-\infty}^{-\frac{y}{2}}\langle x\rangle^{-p\kappa} \langle x+y \rangle^{-q-\frac{\mu}{2}(p-1)} \, dx +\int_{-\frac{y}{2}}^{\infty} \langle x\rangle^{-p\kappa} \langle x+y \rangle^{-q-\frac{\mu}{2}(p-1)} \, dx \\ &\doteq G_1(y)+G_2(y).
\end{align*}

Since on the domain of integration of $G_1(y)$ it holds  $\langle x\rangle \gtrsim \langle y\rangle$ and $\kappa>0$, we get
\begin{align*}
G_1(y)\lesssim \langle y\rangle^{-\kappa} \int _{-\infty}^{-\frac{y}{2}}\langle x\rangle^{-\kappa(p-1)} \langle x+y \rangle^{-q-\frac{\mu}{2}(p-1)} \, dx \lesssim \langle y\rangle^{-\kappa},
\end{align*} where in the last inequality we may use  \eqref{KuKu95 lemma} thanks to \eqref{upper bound for kappa n even} and \eqref{lower bound for kappa n even}.

On the other hand, when $x\geq -\frac{y}{2}$ the inequality $\langle x+y\rangle \gtrsim \langle y \rangle$ is satisfied. Therefore, using again \eqref{KuKu95 lemma}, we find
\begin{align*}
G_2(y) \lesssim \langle y\rangle^{-\kappa} \int_{-\frac{y}{2}}^{\infty} \langle x\rangle^{-p\kappa} \langle x+y \rangle^{\kappa-q-\frac{\mu}{2}(p-1)} \, dx \lesssim \langle y\rangle^{-\kappa}.
\end{align*}

If $y\leq 0$, we get $G(y)=\int_\mathbb{R} \langle x-y\rangle^{-p\kappa} \langle x \rangle^{-q-\frac{\mu}{2}(p-1)} \,dx$. Then, splitting $G(y)$ on  $x\leq \frac{y}{2}$ and on $x\geq \frac{y}{2}$ and proceeding as before, we have the desired estimate.
\end{proof}

\begin{lemma} \label{Lemma 2nd integral n even, square root}
Let us consider $p,\kappa$ satisfying \eqref{p>p0(n+mu1) prel n even}, \eqref{upper bound for kappa n even} and \eqref{lower bound for kappa n even} and let $q$ be defined by \eqref{def q n even} such that $q\geq \tfrac{1}{2}$. Then, we have for any $y\geq 0$
\begin{align}
\int_{-y}^{-\frac{y}{2}} \langle x\rangle^{-\kappa p -\frac{\mu}{2}(p-1)} \frac{\langle x+y \rangle^{-q+\frac{1}{2}}}{\sqrt{x+y}} \,dx \lesssim \langle y\rangle ^{-\kappa}. \label{estimate 2nd integral n even, square root}
\end{align}
\end{lemma}

\begin{proof}
Let $H(y)$ be the integral that appears in the left-hand side of \eqref{estimate 2nd integral n even, square root}. Let us split $H(y)$ in two integrals
\begin{align*}
H(y)&= \int_{-y}^{\tilde{y}} \langle x\rangle^{-\kappa p -\frac{\mu}{2}(p-1)} \frac{\langle x+y \rangle^{-q+\frac{1}{2}}}{\sqrt{x+y}} \,dx  +\int_{\tilde{y}}^{-\frac{y}{2}} \langle x\rangle^{-\kappa p -\frac{\mu}{2}(p-1)} \frac{\langle x+y \rangle^{-q+\frac{1}{2}}}{\sqrt{x+y}} \,dx \\& \doteq H_1(y)+H_2(y),
\end{align*} where $\tilde{y}\doteq \min(-y+1,-\frac{y}{2})$.

For $H_1(y)$, being $\langle x \rangle^{-\frac{\mu}{2}(p-1)}$ and $\langle x+y\rangle^{-q+\frac{1}{2}}$ bounded on $[-y,-y+1]$, we get
\begin{align*}
H_1(y)\lesssim  \int_{-y}^{-y+1}  \frac{\langle x\rangle^{-\kappa p}}{\sqrt{x+y}} \,dx .
\end{align*}
If $y \geq 2$, then, $\langle x\rangle \approx \langle y \rangle$ on $[-y,-y+1]$. Also, because of $p>1$ and $\kappa>0$, we obtain
\begin{align*}
H_1(y)\lesssim \langle y\rangle^{-\kappa p} \int_{-y}^{-y+1}  \frac{dx}{\sqrt{x+y}} \lesssim \langle y\rangle^{-\kappa p} \lesssim \langle y\rangle^{-\kappa }.
\end{align*}
Else, for $0\leq y\leq 2$, since $\langle y\rangle^\kappa$ is bounded, we have 
\begin{align*}
H_1(y)\lesssim  \int_{-y}^{-y+1}  \frac{dx}{\sqrt{x+y}} =2 \lesssim \langle y\rangle^{-\kappa }.
\end{align*} 

Let us estimate $H_2(y)$. Since $\langle x+y\rangle \leq 2(x+y)$ for $x\geq -y+1$, then,
\begin{align*}
H_2(y)\lesssim  \int_{\tilde{y}}^{-\frac{y}{2}} \langle x\rangle^{-\kappa p -\frac{\mu}{2}(p-1)}\langle x+y \rangle^{-q} \,dx \leq \int_{-y}^{-\frac{y}{2}} \langle x\rangle^{-\kappa p -\frac{\mu}{2}(p-1)}\langle x+y \rangle^{-q} \,dx \doteq  \widetilde{H}_2(y).
\end{align*} 
Being $\langle x\rangle \gtrsim \langle x+y\rangle $ and $\langle x \rangle \approx \langle y\rangle $ for $x\in [-y,-\frac{y}{2}]$, we estimate $\widetilde{H}_2(y)$ as follows:
\begin{align*}
\widetilde{H}_2(y) &\lesssim \langle y\rangle^{-\kappa} \int_{-y}^{-\frac{y}{2}} \langle x\rangle^{-\kappa (p-1) -\frac{\mu}{2}(p-1)}\langle x+y \rangle^{-q} \,dx  \\
& \lesssim \langle y\rangle^{-\kappa} \int_{-y}^{-\frac{y}{2}} \langle x+y\rangle^{-\kappa (p-1) -\frac{\mu}{2}(p-1)-q} \,dx  \lesssim  \langle y\rangle^{-\kappa} ,
\end{align*} here we used \eqref{lower bound for kappa n even} in order to guarantee the uniform boundedness of the integral in the last line. Combining the estimates for $H_1(y)$ and $H_2(y)$, we find \eqref{estimate 2nd integral n even, square root}. 
\end{proof}

\begin{lemma} \label{Lemma 2nd integral n even, square root q<1/2}
Let us consider $p,\kappa$ satisfying \eqref{p>p0(n+mu1) prel n even}, \eqref{upper bound for kappa n even} and \eqref{lower bound for kappa n even} and let $q$ be defined by \eqref{def q n even} such that $0\leq q< \tfrac{1}{2}$. Then, we have for any $y\geq 0$
\begin{align}
\int_{-y}^{-\frac{y}{2}} \langle x\rangle^{-\kappa p -\frac{\mu}{2}(p-1)+\frac{1}{2}} \frac{\langle x+y \rangle^{-q}}{\sqrt{x+y}} \,dx \lesssim \langle y\rangle ^{-\kappa}. \label{estimate 2nd integral n even, square root q<1/2}
\end{align}
\end{lemma}

\begin{proof}
We have to modify slightly the proof of Lemma \ref{Lemma 2nd integral n even, square root}. Let $I(y)$ be the integral that appears in the left-hand side of \eqref{estimate 2nd integral n even, square root q<1/2}. Even in this case we split $I(y)$ in two integrals
\begin{align*}
I(y)&= \int_{-y}^{\tilde{y}} \langle x\rangle^{-\kappa p -\frac{\mu}{2}(p-1)+\frac{1}{2}} \frac{\langle x+y \rangle^{-q}}{\sqrt{x+y}} \,dx  +\int_{\tilde{y}}^{-\frac{y}{2}} \langle x\rangle^{-\kappa p -\frac{\mu}{2}(p-1)+\frac{1}{2}} \frac{\langle x+y \rangle^{-q}}{\sqrt{x+y}} \,dx \\& \doteq I_1(y)+I_2(y),
\end{align*} where $\tilde{y}=\min(-y+1,-\frac{y}{2})$ as before.

We begin with $I_1(y)$. Since $\langle x+y\rangle^{-q}$ is bounded on $[-y,-y+1]$, it holds
\begin{align*}
I_1(y)\lesssim  \int_{-y}^{-y+1}  \frac{\langle x\rangle^{-\kappa p -\frac{\mu}{2}(p-1)+\frac{1}{2}}}{\sqrt{x+y}} \,dx .
\end{align*}
If $y \geq 2$, then, $\langle x\rangle \approx \langle y \rangle$ on $[-y,-y+1]$. Also,
\begin{align*}
I_1(y)\lesssim \langle y\rangle^{-\kappa } \int_{-y}^{-y+1}  \frac{\langle x\rangle^{-\kappa (p-1) -\frac{\mu}{2}(p-1)+\frac{1}{2}}}{\sqrt{x+y}} \, dx \lesssim \langle y\rangle^{-\kappa } \int_{-y}^{-y+1}  \frac{dx}{\sqrt{x+y}} \lesssim \langle y\rangle^{-\kappa }.
\end{align*} where in the second last inequality we used the fact that the exponent of $\langle x\rangle$ is nonpositive. Indeed, $-\kappa (p-1) -\tfrac{\mu}{2}(p-1)+\tfrac{1}{2}\leq 0$ is equivalent to require 
\begin{align} \label{lower bound kappa n even, case q<1/2}
\kappa\geq \tfrac{1}{2(p-1)}-\tfrac{\mu}{2}.
\end{align}
 But thanks to the assumption $q<\tfrac{1}{2}$ this lower bound on $\kappa$ is weaker than the lower bound in \eqref{lower bound for kappa n even}. Therefore, under the assumptions we are working with, \eqref{lower bound kappa n even, case q<1/2} is always satisfied.
On the other hand, for $0\leq y\leq 2$, since $\langle x\rangle$ is bounded on $[-y,-y+1]$ and $\langle y\rangle^\kappa$ is bounded as well, we have 
\begin{align*}
I_1(y)\lesssim  \int_{-y}^{-y+1}  \frac{dx}{\sqrt{x+y}} =2 \lesssim \langle y\rangle^{-\kappa }.
\end{align*} 

Using again \eqref{lower bound kappa n even, case q<1/2}, it is possible to show the estimate $I_2(y)\lesssim \langle y \rangle^{-\kappa}$ exactly as we have done in Lemma \ref{Lemma 2nd integral n even, square root} for the term $H_2(y)$.
%

Summarizing, the estimates for $I_1(y)$ and $I_2(y)$ imply  \eqref{estimate 2nd integral n even, square root q<1/2}. 
\end{proof}

\begin{lemma} \label{Lemma 2nd integral n even, square root q<0}
Let us consider $p,\kappa$ satisfying \eqref{p>p0(n+mu1) prel n even}, \eqref{upper bound for kappa n even} and \eqref{lower bound for kappa n even} and let $q$ be defined by \eqref{def q n even} such that $-\tfrac{1}{2}\leq  q<0$. Then, we have for any $y\geq 0$
\begin{align}
\int_{-y}^{-\frac{y}{2}} \langle x\rangle^{-\kappa p -\frac{\mu}{2}(p-1)+1} \frac{\langle x+y \rangle^{-q-\frac{1}{2}}}{\sqrt{x+y}} \,dx \lesssim \langle y\rangle ^{-\kappa}. \label{estimate 2nd integral n even, square root q<0}
\end{align}
\end{lemma}

\begin{proof}
First of all, we point out that $\langle x\rangle \approx \langle y \rangle$ on the domain of integration. Therefore, if we denote by $J(y)$ the integral in the left-hand side of \eqref{estimate 2nd integral n even, square root q<0}, then, 
\begin{align*}
J(y)\lesssim \langle y\rangle^{-\kappa} \int_{-y}^{-\frac{y}{2}} \langle x\rangle^{-\kappa (p-1) -\frac{\mu}{2}(p-1)+1} \frac{\langle x+y \rangle^{-q-\frac{1}{2}}}{\sqrt{x+y}} \,dx.
\end{align*}
Using \eqref{lower bound for kappa n even}, it results
\begin{align*}
\langle x\rangle^{-\kappa (p-1) -\frac{\mu}{2}(p-1)+1} \leq \langle x\rangle^{q-\varepsilon},
\end{align*} for a suitably small $\varepsilon>0$. Also,
\begin{align*}
J(y)\lesssim \langle y\rangle^{-\kappa} \int_{-y}^{-\frac{y}{2}}  \langle x\rangle^{q-\varepsilon} \frac{\langle x+y \rangle^{-q-\frac{1}{2}}}{\sqrt{x+y}} \,dx.
\end{align*}
Since $\langle x+y\rangle \lesssim \langle x\rangle$ on $[-y,-\frac{y}{2}]$ and $q<0$, we find $\langle x+y\rangle^{-q} \lesssim \langle x\rangle^{-q}$, which implies 
\begin{align*}
J(y)\lesssim \langle y\rangle^{-\kappa} \int_{-y}^{-\frac{y}{2}}  \langle x\rangle^{-\varepsilon} \frac{\langle x+y \rangle^{-\frac{1}{2}}}{\sqrt{x+y}} \,dx.
\end{align*}

The last step is to prove the uniform boundedness of the $x-$integral in the right-hand side of the previous inequality. Integration by parts leads to 
\begin{align*}
\int_{-y}^{-\frac{y}{2}}  \langle x\rangle^{-\varepsilon} \frac{\langle x+y \rangle^{-\frac{1}{2}}}{\sqrt{x+y}} \,dx & \leq \int_{-y}^{\infty}  \langle x\rangle^{-\varepsilon} \frac{\langle x+y \rangle^{-\frac{1}{2}}}{\sqrt{x+y}} \,dx \\
& \lesssim \int_{-y}^{\infty}   \sqrt{x+y} \big(\langle x\rangle^{-1-\varepsilon}\langle x+y \rangle^{-\frac{1}{2}}+\langle x\rangle^{-\varepsilon}\langle x+y \rangle^{-\frac{3}{2}}\big) \,dx
\\
& \lesssim \int_{-y}^{\infty}  \big(\langle x\rangle^{-1-\varepsilon}+\langle x\rangle^{-\varepsilon}\langle x+y \rangle^{-1}\big) \,dx \, \lesssim 1,
\end{align*} where in the last inequality we used \eqref{KuKu95 lemma}. This concludes the proof.
\end{proof}

\begin{rmk} \label{rmk q>-1/2} Comparing the statements of Lemmas \ref{Lemma 2nd integral n even, square root}, \ref{Lemma 2nd integral n even, square root q<1/2} and \ref{Lemma 2nd integral n even, square root q<0}, we see that we have estimated suitable integrals for $q\geq -\tfrac{1}{2}$. The condition $q\geq -\tfrac{1}{2}$ is equivalent to require $p\geq \tfrac{n}{n-1}$. Of course, we want to keep the lower bound in \eqref{p>p0(n+mu1) prel n even} as main lower bound for $p$. Therefore, we have to guarantee that $$\tfrac{n}{n-1}\leq p_0(n+\mu).$$ It turns out that such a condition is equivalent to require $$\mu\leq \widetilde{M}(n)\doteq \tfrac{3n^2-5n+2}{n}.$$ 
 
 In the upcoming results we will require a stronger upper bound for $\mu$, so that, the above condition on $\mu$ will be every time fulfilled and, in turn, the condition $q\geq -\tfrac{1}{2}$ will be valid as well.
\end{rmk}

\begin{lemma} \label{Lemma 3rd integral n even, 3 factors}
Let us consider $p,\kappa$ satisfying \eqref{p>p0(n+mu1) prel n even}, \eqref{upper bound for kappa n even} and \eqref{lower bound for kappa n even} and let $q$ be defined by \eqref{def q n even}. Then, we have for any $y\geq 0$
\begin{align}
\int_{-2y}^{y} \langle x-y\rangle^{-\frac{\mu}{2}(p-1)} \langle x+2y \rangle^{-q-1} \langle x\rangle^{-\kappa p}\,dx \lesssim \langle y\rangle ^{-\kappa}. \label{estimate 3rd integral n even, 3 factors}
\end{align}
\end{lemma}

\begin{proof}
Let $K(y)$ denotes the integral on the left-hand side of \eqref{estimate 3rd integral n even, 3 factors}. We split the integral in two parts
\begin{align*}
K(y)&=\int_{-2y}^{-\frac{y}{2}} \langle x-y\rangle^{-\frac{\mu}{2}(p-1)} \langle x+2y \rangle^{-q-1} \langle x\rangle^{-\kappa p}\,dx \\ & \quad +\int_{-\frac{y}{2}}^{y} \langle x-y\rangle^{-\frac{\mu}{2}(p-1)} \langle x+2y \rangle^{-q-1} \langle x\rangle^{-\kappa p}\,dx \doteq K_1(y)+K_2(y).
\end{align*}

On the one hand, we can use the relations $\langle x\rangle\approx \langle y\rangle$ and $\langle x+2y\rangle \leq \langle x-y\rangle$ when $x\in[-2y,-\frac{y}{2}]$, obtaining for $K_1(y)$
\begin{align*}
K_1(y)\lesssim \langle y\rangle^{-\kappa}\int_{-2y}^{-\frac{y}{2}}\langle x+2y \rangle^{-q-1-\frac{\mu}{2}(p-1)} \langle x\rangle^{-\kappa (p-1)}\,dx \lesssim \langle y\rangle^{-\kappa},
\end{align*} where in the last inequality we can apply \eqref{KuKu95 lemma} because of \eqref{lower bound for kappa n even}.
On the other hand, since $\langle x+2y\rangle \geq \langle x-y\rangle,\langle x\rangle$ for $x\in [-\frac{y}{2},y]$, employing again \eqref{lower bound for kappa n even}, we find 
\begin{align*}
K_2(y)\lesssim  \int_{-\frac{y}{2}}^{y}\langle x+2y \rangle^{-\frac{\mu}{2}(p-1)-q-1-\kappa p} \,dx \lesssim  \langle y \rangle^{-\frac{\mu}{2}(p-1)-q-\kappa p } \lesssim \langle y\rangle^{-\kappa}.
\end{align*}
The desired estimate follows from the estimates for $K_1(y)$ and $K_2(y)$. 
\end{proof}

Let us introduce four auxiliary integrals, which will come into play in the treatment of the semilinear problem. Let $t\geq 0,r>0$ and let $\gamma$ be $0$ or $\frac{1}{2}$, we define
\begin{align}
I_\gamma(t,r) &\doteq \int_0^t \langle \tau \rangle^{-\frac{\mu}{2}(p-1)} \int_{|\lambda_-|}^{\lambda_+}\langle \lambda\rangle^{-q+\frac{p}{2}-\frac{1}{2}-\gamma} \frac{\phi_\kappa(\tau,\lambda)^p}{\sqrt{\lambda-\lambda_-}} \, d\lambda \, d\tau , \label{def Igamma n even}\\
J_\gamma(t,r) &\doteq \int_0^{(t-r)_+} \langle \tau \rangle^{-\frac{\mu}{2}(p-1)}\langle \lambda_-\rangle^{-q-\frac{1}{2}-\gamma}  \int^{\lambda_-}_0\frac{ \langle\lambda\rangle^{\frac{p}{2}} \phi_\kappa(\tau,\lambda)^p}{\sqrt{\lambda_--\lambda}} \, d\lambda \, d\tau , \label{def Jgamma n even}\\
P_\gamma(t,r) &\doteq \int_0^{(t-r)_+} \langle \tau \rangle^{-\frac{\mu}{2}(p-1)}\langle \lambda_-\rangle^{-q+\frac{p}{2}-1-\gamma}   \phi_\kappa(\tau,\tfrac{\lambda_-}{2})^p \, d\tau , \label{def Pgamma n even}\\
Q_\gamma(t,r) &\doteq \int_{(t-r)_+}^t \langle \tau \rangle^{-\frac{\mu}{2}(p-1)}\langle \lambda_-\rangle^{-q+\frac{p}{2}-1-\gamma}   \phi_\kappa(\tau,-\lambda_-)^p \, d\tau , \label{def Qgamma n even}
\end{align} where we have set $\lambda_\pm\doteq t-\tau\pm r$ , $q$ is defined by \eqref{def q n even} and $\phi_\kappa$ is  given by \eqref{def phi kappa}. 

The next proposition provides some estimates for the above defined integrals. Let us underline explicitly that the core of the proof of  Theorem \ref{thm GER n even semi} is the next result.

\begin{prop} \label{Lemma I,J,P,Q gamma=0,1/2} 
Let us consider an even integer $n\geq 4$ and $p,\kappa$ satisfying \eqref{p>p0(n+mu1) prel n even}, \eqref{upper bound for kappa n even} and \eqref{lower bound for kappa n even} and let $q$ be defined by \eqref{def q n even} such that 
\begin{align} \label{q condition n even}
 -\tfrac{1}{2}\leq \, &q\leq m-\tfrac{1}{2},\\
&p< p_{\Fuj}(\mu). \label{p<p Fuj (mu1)}
\end{align} 

 Then, the following estimates are fulfilled for any $t\geq 0,r>0$ and $\gamma\in \{0,\frac{1}{2}\}$:
\begin{align}
I_\gamma(t,r) & \lesssim \langle t-r\rangle^{-\kappa-\gamma}, \label{Igamma est n even} \\
J_\gamma(t,r) & \lesssim \langle t-r\rangle^{-\kappa-\gamma}, \label{Jgamma est n even} \\
P_\gamma(t,r) & \lesssim \langle t-r\rangle^{-\kappa-\gamma}, \label{Pgamma est n even} \\
Q_\gamma(t,r) & \lesssim \langle t-r\rangle^{-\kappa-\gamma}. \label{Qgamma est n even}
\end{align}
\end{prop}

\begin{rmk}\label{rmk after prop I,J,P,Q gamma=0,1/2} Let us analyze all assumptions we have done in the previous statement. The condition from above on $q$ in \eqref{q condition n even} is equivalent to $p\leq 2$ and, therefore, it is always fulfilled in our setting. Indeed, we are assuming $p<p_{\Fuj}\big(\frac{n+\mu-1}{2}\big)$ and this implies $p<2$ for $n\geq 4$ and $\mu\geq 2$. 
On the other hand, the condition on $p$ prescribed by \eqref{p<p Fuj (mu1)} requires the validity of $p_0(n+\mu)<p_{\Fuj}(\mu),$ in order to have a nonempty range for $p$.  This condition is valid for $\mu< M(n)$, where $M(n)$ is defined by \eqref{definition M1(n)}. Indeed,  $p_0(n+\mu)<p_{\Fuj}(\mu)$ is equivalent to require $$(n+\mu-1)p_{\Fuj}(\mu)^2-(n+\mu+1)p_{\Fuj}(\mu)-2 =-\tfrac{1}{\mu^2}\big(\mu^2 -(n-1)\mu-2(n-1)\big)>0,$$ which is obviously satisfied for $0\leq \mu <M(n)$. Furthermore, as we said in Remark \ref{rmk q>-1/2}, for $n\geq 4$ it holds $M(n)<\widetilde{M}(n)$, so that, $q\geq -\tfrac{1}{2}$ is valid for $\mu\in[2,M(n))$ and, then, the lower bound for $q$ in \eqref{q condition n even} is satisfied.
\end{rmk}

\begin{rmk}\label{rmk after thm GER n even} Thanks to Remark \ref{rmk after prop I,J,P,Q gamma=0,1/2} we can clarify now the choice of the parameters $\kappa_1$, $\kappa_2$ and $\mu$ in Theorem \ref{thm GER n even semi}. Indeed, as in  \eqref{upper bound for kappa n even} and in \eqref{lower bound for kappa n even}, in the statement of Theorem \ref{thm GER n even semi}, we have
\begin{align*}
\kappa_1 &\doteq \tfrac{-q+1}{p-1}-\tfrac{\mu}{2}=\tfrac{2}{p-1}-\tfrac{n+\mu-1}{2};\\
\kappa_2 &\doteq  q+\tfrac{\mu}{2}(p-1)=\tfrac{n+\mu-1}{2}(p-1)-1.
\end{align*}

As we have already seen, $(\kappa_1,\kappa_2]$ is not empty because of the condition $p>p_0(n+\mu)$ and $\kappa_1>0$ thanks to the upper bound $p<p_{\Fuj}\big(\tfrac{n+\mu-1}{2}\big)$. 
Besides, we can find $\kappa\in(\kappa_1,\kappa_2]$ such that $k<m+\frac{1}{2}$, since $\kappa_2<m+\frac{1}{2}$ is equivalent to require $p< \tfrac{2n+\mu}{n+\mu-1}$. However, this upper bound for $p$ is weaker than the upper bound for $p$ in \eqref{p>p0(n+mu1) prel n even}. Hence, for the range of $p$s considered in the statement of Theorem \ref{thm GER n even semi} the inequality $\kappa_2<m+\frac{1}{2}$ is always satisfied. Finally, as we have just explained in Remark \ref{rmk after prop I,J,P,Q gamma=0,1/2}, the upper bound for $\mu$ is due to the fact that we want to guarantee the validity of the condition $p_0(n+\mu)<p_{\Fuj}(\mu)$, which implies a not empty range of admissible values for $p$ in Theorem \ref{thm GER n even semi}, while the lower bound $\mu\geq 2$ is a necessary condition coming from \eqref{delta=1 condition} for nontrivial and nonnegative $\mu$ and $\nu$. 
\end{rmk}

\begin{proof}[Proof of Proposition \ref{Lemma I,J,P,Q gamma=0,1/2}]
 We will modify the proof of Proposition 6.6 in \cite{KuKu96}, by using the previously derived lemmas. Let us start with $I_\gamma(t,r)$. 
Since
\begin{align*}
I_\gamma(t,r) &= \int_0^t \langle \tau \rangle^{-\frac{\mu}{2}(p-1)} \int_{|\lambda_-|}^{\lambda_+}\langle \lambda\rangle^{-q+\frac{p}{2}-\frac{1}{2}-\gamma} \langle \lambda-\tau \rangle^{-\kappa p}\frac{\langle \lambda+\tau \rangle^{-\frac{p}{2}}}{\sqrt{\lambda-\lambda_-}} \, d\lambda \, d\tau 
\end{align*} performing the change of variables $\xi=\lambda+\tau,\eta=\lambda-\tau$, we get
\begin{align*}
I_\gamma(t,r) &\lesssim  \int_{|t-r|}^{t+r} \frac{\langle \xi \rangle^{-\frac{p}{2}}}{\sqrt{\xi+r-t}}  \int_{-\xi}^{\xi} \langle \xi-\eta \rangle^{-\frac{\mu}{2}(p-1)}\langle \xi+\eta \rangle^{-q+\frac{p}{2}-\frac{1}{2}-\gamma} \langle \eta \rangle^{-\kappa p} \, d\eta \, d\xi.
\end{align*}
Let us estimate the $\eta-$integral. We split the domain of integration in three subintervals.
\begin{align*}
\int_{-\xi}^{\xi} &\langle \xi-\eta \rangle^{-\frac{\mu}{2}(p-1)}\langle \xi+\eta \rangle^{-q+\frac{p}{2}-\frac{1}{2}-\gamma} \langle \eta \rangle^{-\kappa p} \, d\eta = A_1(\xi)+A_2(\xi)+A_3(\xi),
\end{align*} where $A_1(\xi)$, $A_2(\xi)$ and $A_3(\xi)$ denote the integrals of the integrand in the left-hand side over $\big[\tfrac{\xi}{2},\xi\big]$, $\big[-\tfrac{\xi}{2},\tfrac{\xi}{2}\big]$ and $\big[-\xi,-\tfrac{\xi}{2}\big]$, respectively.

Let us begin with $A_1(\xi)$. Since $\langle \xi +\eta \rangle \approx \langle \eta\rangle \approx \langle \xi \rangle$ for $\eta\in[\tfrac{\xi}{2},\xi]$, using \eqref{p<p Fuj (mu1)}, we have
\begin{align*}
A_1(\xi)& \lesssim \langle \xi\rangle^{-q+\frac{p}{2}-\frac{1}{2}-\gamma -\kappa p} \int_{\frac{\xi}{2}}^{\xi} \langle \xi-\eta \rangle^{-\frac{\mu}{2}(p-1)}\, d\eta  \lesssim \langle \xi\rangle^{-q+\frac{p}{2}-\frac{1}{2}-\gamma -\kappa p -\frac{\mu}{2}(p-1)+1} .
\end{align*}

We estimate now $A_2(\xi)$. Being $\langle \xi-\eta\rangle \approx \langle \xi+\eta\rangle \approx \langle \xi\rangle$ for $\eta\in [-\tfrac{\xi}{2},\tfrac{\xi}{2}]$, we have
\begin{align*}
A_2(\xi)& \lesssim  \langle \xi\rangle^{-\frac{\mu}{2}(p-1)-q+\frac{p}{2}-\frac{1}{2}-\gamma} \int^{\frac{\xi}{2}}_{-\frac{\xi}{2}} \langle \eta \rangle^{-\kappa p} \, d\eta \\ & \lesssim 
\begin{cases}
\langle \xi\rangle^{-\frac{\mu}{2}(p-1)-q+\frac{p}{2}-\frac{1}{2}-\gamma-\kappa p+1} & \mbox{if} \quad -\kappa p+1>0 , \\
\langle \xi\rangle^{-\frac{\mu}{2}(p-1)-q+\frac{p}{2}-\frac{1}{2}-\gamma-\kappa p+1+\varepsilon} & \mbox{if} \quad -\kappa p+1=0, \\
\langle \xi\rangle^{-\frac{\mu}{2}(p-1)-q+\frac{p}{2}-\frac{1}{2}-\gamma} & \mbox{if} \quad -\kappa p+1<0,
\end{cases}
\end{align*} where $\varepsilon>0$ in the logarithmic case can be chosen sufficiently small so that  
\begin{align}\label{lower bound kappa with epsilon}
-\tfrac{\mu}{2}(p-1)-q-\kappa(p-1)+1+\varepsilon<0,
\end{align}
 thanks to \eqref{lower bound for kappa n even}.

Eventually, we consider $A_3(\xi)$. On $[-\xi,-\tfrac{\xi}{2}]$ we have $\langle \xi-\eta\rangle\approx \langle \eta\rangle\approx\langle \xi \rangle$, also,
\begin{align*}
A_3(\xi) & \lesssim  \langle \xi\rangle^{-\frac{\mu}{2}(p-1)-\kappa p}  \int^{-\frac{\xi}{2}}_{-\xi}  \langle \xi+\eta \rangle^{-q+\frac{p}{2}-\frac{1}{2}-\gamma} \, d\eta \lesssim  \langle \xi\rangle^{-\frac{\mu}{2}(p-1)-\kappa p-q+\frac{p}{2}+\frac{1}{2}-\gamma} ,
\end{align*} where in the last inequality we used $-q+\frac{p}{2}-\frac{1}{2}-\gamma>-1$, which is equivalent to $p<p_{\Fuj}(\tfrac{n-2}{2})$. This condition on $p$ is always fulfilled thanks to the upper bound in \eqref{p>p0(n+mu1) prel n even}.
Combining the estimates for $A_1(\xi),A_2(\xi)$ and $A_3(\xi)$, it results
\begin{align*}
I_\gamma(t,r) &\lesssim\! \int_{|t-r|}^{t+r} \! \frac{\langle \xi \rangle^{-\frac{\mu}{2}(p-1)-q-\frac{1}{2}-\gamma+\alpha(\kappa)}}{\sqrt{\xi+r-t}} d\xi \lesssim \langle t-r\rangle^{-\kappa -\gamma}\! \! \int_{|t-r|}^{t+r} \!\frac{\langle \xi \rangle^{-\frac{\mu}{2}(p-1)-q-\frac{1}{2}+\kappa +\alpha(\kappa)}}{\sqrt{\xi+r-t}} d\xi ,
\end{align*} where $\alpha(\kappa)=-\kappa p+1$ if $\kappa<\tfrac{1}{p}$,  $\alpha(\kappa)=-\kappa p+1+\varepsilon$ if $\kappa=\tfrac{1}{p}$ and  $\alpha(\kappa)=0$ if $\kappa>\tfrac{1}{p}$. We point out that the power for $\langle\xi\rangle$ in the last integral can be written in all three subcases as $-\beta(\kappa)-\tfrac{1}{2}$ for a positive constant $\beta(\kappa)$, due to \eqref{lower bound for kappa n even}, \eqref{lower bound kappa with epsilon} and \eqref{upper bound for kappa n even}. Therefore,
\begin{align*}
I_\gamma(t,r) &\lesssim \langle t-r\rangle^{-\kappa -\gamma}  \int_{|t-r|}^{t+r} \frac{\langle \xi \rangle^{-\beta(\kappa)-\tfrac{1}{2}}}{\sqrt{\xi+r-t}} \,d\xi  \lesssim \langle t-r\rangle^{-\kappa -\gamma} .
\end{align*} 
Indeed, using integration by parts, for $t\geq r$ we may estimate the $\xi-$integral as follows:
\begin{align*}
\int_{|t-r|}^{t+r} \frac{\langle \xi \rangle^{-\beta(\kappa)-\tfrac{1}{2}}}{\sqrt{\xi+r-t}} \,d\xi & \leq\int_{t-r}^{\infty} \frac{\langle \xi \rangle^{-\beta(\kappa)-\tfrac{1}{2}}}{\sqrt{\xi+r-t}} \,d\xi  \lesssim \int_{t-r}^{\infty} \sqrt{\xi+r-t} \langle \xi \rangle^{-\beta(\kappa)-\tfrac{3}{2}} \,d\xi \\ &\lesssim \int_{t-r}^{\infty} \langle \xi \rangle^{-\beta(\kappa)-1} \,d\xi \, \lesssim 1.
\end{align*} 
On the other hand, employing again integration by parts, for $t\leq r$ we get
\begin{align*}
\int_{r-t}^{t+r} \frac{\langle \xi \rangle^{-\beta(\kappa)-\tfrac{1}{2}}}{\sqrt{\xi+r-t}} \,d\xi & \leq \int_{r-t}^{\infty} \frac{\langle \xi \rangle^{-\beta(\kappa)-\tfrac{1}{2}}}{\sqrt{\xi+r-t}} \,d\xi \\ & \lesssim \sqrt{r-t} \langle r-t\rangle^{-\beta(\kappa)-\frac{1}{2}}+\int_{r-t}^{\infty} \sqrt{\xi+r-t} \langle \xi \rangle^{-\beta(\kappa)-\tfrac{3}{2}} \,d\xi \\ & \lesssim \langle r-t\rangle^{-\beta(\kappa)}+\int_{r-t}^{\infty}  \langle \xi \rangle^{-\beta(\kappa)-1} \,d\xi \,\lesssim 1.
\end{align*} 



Let us estimate $J_\gamma(t,r)$. We can assume $t> r$. Carrying out the same change of variable we used for $I_\gamma(t,r)$, we get
\begin{align*}
J_\gamma(t,r) &= \int_0^{t-r} \langle \tau \rangle^{-\frac{\mu}{2}(p-1)}\langle \lambda_-\rangle^{-q-\frac{1}{2}-\gamma}  \int^{\lambda_-}_0\frac{ \langle\lambda\rangle^{\frac{p}{2}} \langle \lambda+\tau\rangle^{-\frac{p}{2}}\langle \lambda-\tau\rangle^{-\kappa p}}{\sqrt{\lambda_--\lambda}} \, d\lambda \, d\tau \\
&\lesssim  \int_0^{t-r} \!\frac{\langle\xi\rangle^{-\frac{p}{2}}}{\sqrt{t-r-\xi}}\int_{-\xi}^{\xi}\langle \xi -\eta\rangle^{-\frac{\mu}{2}(p-1)} \big\langle t-r+\tfrac{\eta-\xi}{2}\big\rangle^{-q-\frac{1}{2}-\gamma}\langle\xi+\eta\rangle^{\frac{p}{2}}\langle \eta\rangle^{-\kappa p}\,d\eta\, d\xi. 
\end{align*}
Let us split the domain of integration in the following three regions:
\begin{align*}
\Omega_1 &= \big\{(\xi,\eta)\in \mathbb{R}^2:  0<\xi<t-r, -\tfrac{\xi}{2}<\eta<\xi\big\}, \\
\Omega_2 &= \big\{(\xi,\eta)\in \mathbb{R}^2:  0<\xi< \tfrac{t-r}{2}, -\xi<\eta<-\tfrac{\xi}{2} \big\}, \\
\Omega_3 &=\big\{(\xi,\eta)\in \mathbb{R}^2: \tfrac{t-r}{2}<\xi<t-r, -\xi<\eta<  -\tfrac{\xi}{2}\big\}.
\end{align*}  
Thus, we can write $J_\gamma(t,r)=J_{\gamma,1}(t,r)+J_{\gamma,2}(t,r)+J_{\gamma,3}(t,r),$
where $J_{\gamma,k}(t,r)$ is the integral over $\Omega_k$ for $k=1,2,3$. We begin with $J_{\gamma,1}(t,r)$. Since on $\Omega_1$ we have
\begin{equation}
\big\langle t-r+\tfrac{\eta-\xi}{2}\big\rangle \approx \langle t-r \rangle, \quad
\big\langle t-r+\tfrac{\eta-\xi}{2}\big\rangle  \gtrsim \langle \xi+\eta \rangle ,\label{conditions in Omega1}
\end{equation} then, being $\kappa-q-\frac{\mu}{2}(p-1)\leq 0$ because of \eqref{upper bound for kappa n even},  we have 
\begin{align*}
\big\langle t-r+\tfrac{\eta-\xi}{2}\big\rangle^{-q-\frac{1}{2}-\gamma} \lesssim  \langle t-r \rangle^{-\kappa-\gamma-\frac{1}{2}+\frac{\mu}{2}(p-1)} \langle \xi+\eta \rangle^{\kappa-q-\frac{\mu}{2}(p-1)}.
\end{align*}
Besides, $\langle \xi+\eta \rangle \lesssim \langle \xi \rangle$ implies  that $\langle \xi+\eta \rangle^{\frac{p}{2}} \langle \xi \rangle^{-\frac{p}{2}}  $ is bounded on the domain of integration. Also, using $\langle \xi+\eta\rangle\approx \langle\xi \rangle $ for $\eta\in [-\tfrac{\xi}{2},\xi]$, it follows:
\begin{align*}
\langle t-r \rangle^{\kappa+\gamma+\frac{1}{2}-\frac{\mu}{2}(p-1)}J_{\gamma,1}& (t, r)  \lesssim  \iint_{\Omega_1}  \frac{\langle \xi -\eta\rangle^{-\frac{\mu}{2}(p-1)} \langle \xi+\eta \rangle^{\kappa-q-\frac{\mu}{2}(p-1)}}{\sqrt{t-r-\xi}}\langle \eta\rangle^{-\kappa p}\,d\eta\, d\xi \\
& \lesssim  \int_{0}^{t-r} \frac{\langle \xi\rangle^{-\frac{\mu}{2}(p-1)}}{\sqrt{t-r-\xi}} \int_{-\frac{\xi}{2}}^{\xi} \langle \xi -\eta\rangle^{-\frac{\mu}{2}(p-1)}   \langle \xi+\eta \rangle^{\kappa-q}\langle \eta\rangle^{-\kappa p}\,d\eta\, d\xi.
\end{align*}
Now we show that the $\eta-$integral in the last line of the previous estimate is uniformly bounded. Since $\langle \eta+\xi\rangle\approx \langle \xi-\eta\rangle$ for $\eta\in[-\tfrac{\xi}{2},\tfrac{\xi}{2}]$ and $\langle \eta+\xi\rangle\approx \langle \eta\rangle$ for $\eta\in[\tfrac{\xi}{2},\xi]$, then, using \eqref{KuKu95 lemma}, we obtain
\begin{align*}
\int_{-\frac{\xi}{2}}^{\xi} \langle \xi -\eta\rangle^{-\frac{\mu}{2}(p-1)}   \langle \xi+\eta \rangle^{\kappa-q}\langle \eta\rangle^{-\kappa p}\,d\eta  
&\lesssim \int_{-\frac{\xi}{2}}^{\frac{\xi}{2}} \langle \xi -\eta\rangle^{-\frac{\mu}{2}(p-1)+\kappa-q}  \langle \eta\rangle^{-\kappa p}\,d\eta \\ & \qquad + \int_{\frac{\xi}{2}}^{\xi} \langle \xi -\eta\rangle^{-\frac{\mu}{2}(p-1)}  \langle \eta\rangle^{-\kappa (p-1)-q}\,d\eta \lesssim 1.
\end{align*}
In particular, in the last estimates we used that the exponent for $\langle \xi -\eta \rangle$ in the first integral is nonnegative thanks to \eqref{upper bound for kappa n even} and that the exponent of $\langle\eta\rangle$ in the second integral is smaller than $\frac{\mu}{2}(p-1)-1$, due to \eqref{lower bound for kappa n even}, and, then, smaller than 0 thanks to the assumption \eqref{p<p Fuj (mu1)}. 
Thus, it follows:
\begin{align*}
J_{\gamma,1}(t,r) & \lesssim \langle t-r \rangle^{-\kappa-\gamma-\frac{1}{2}+\frac{\mu}{2}(p-1)} \int_{0}^{t-r} \frac{\langle \xi\rangle^{-\frac{\mu}{2}(p-1)}}{\sqrt{t-r-\xi}} \, d\xi.
\end{align*} 
For $t-r\geq 1$, we may estimate the $\xi-$integral in the following way:
\begin{align*}
\int_{0}^{t-r} &\frac{\langle \xi\rangle^{-\frac{\mu}{2}(p-1)}}{\sqrt{t-r-\xi}} \, d\xi =\int_{0}^{\frac{t-r}{2}} \frac{\langle \xi\rangle^{-\frac{\mu}{2}(p-1)}}{\sqrt{t-r-\xi}} \, d\xi +\int_{\frac{t-r}{2}}^{t-r} \frac{\langle \xi\rangle^{-\frac{\mu}{2}(p-1)}}{\sqrt{t-r-\xi}} \, d\xi \\
&\qquad\lesssim  (\sqrt{t-r})^{-1}\int_{0}^{\frac{t-r}{2}} \langle \xi\rangle^{-\frac{\mu}{2}(p-1)} \, d\xi +\langle t-r \rangle^{-\frac{\mu}{2}(p-1)}\int_{\frac{t-r}{2}}^{t-r} \frac{d\xi}{\sqrt{t-r-\xi}} \\
&\qquad\lesssim  \langle t-r\rangle^{-\frac{1}{2}}\langle t-r\rangle^{-\frac{\mu}{2}(p-1)+1} +\langle t-r \rangle^{-\frac{\mu}{2}(p-1)}\sqrt{t-r}  \lesssim \langle t-r\rangle^{-\frac{\mu}{2}(p-1)+\frac{1}{2}} .
\end{align*}
Otherwise, if $0<t-r<1$, then, using the fact that $\langle t-r\rangle \approx 1$, we get immediately
\begin{align*}
\int_{0}^{t-r} \frac{\langle \xi\rangle^{-\frac{\mu}{2}(p-1)}}{\sqrt{t-r-\xi}} \, d\xi  \leq \int_{0}^{t-r} \frac{ d\xi }{\sqrt{t-r-\xi}}  \approx \sqrt{t-r} \leq \langle t-r \rangle^{\frac{1}{2}} \approx \langle t-r \rangle^{\frac{1}{2}-\frac{\mu}{2}(p-1)}.
\end{align*}
Summarizing, we got  $J_{\gamma,1}(t,r)  \lesssim \langle t-r \rangle^{-\kappa-\gamma}$.

Similarly, we can now estimate $J_{\gamma,2}(t,r) $. Indeed, since \eqref{conditions in Omega1} is valid also in $\Omega_2$, proceeding as before, we find
\begin{align*}
J_{\gamma,2}(t,r) & \lesssim  \langle t-r \rangle^{-\kappa-\gamma-\frac{1}{2}+\frac{\mu}{2}(p-1)} \iint_{\Omega_2}  \frac{\langle \xi -\eta\rangle^{-\frac{\mu}{2}(p-1)} \langle \xi+\eta \rangle^{\kappa-q-\frac{\mu}{2}(p-1)}}{\sqrt{t-r-\xi}}\langle \eta\rangle^{-\kappa p}\,d\eta\, d\xi \\
& \lesssim \langle t-r \rangle^{-\kappa-\gamma-\frac{1}{2}+\frac{\mu}{2}(p-1)} \int_{0}^{\frac{t-r}{2}} \frac{\langle \xi\rangle^{-\frac{\mu}{2}(p-1)}}{\sqrt{t-r-\xi}} \int^{-\frac{\xi}{2}}_{-\xi}  \langle \xi+\eta \rangle^{\kappa-q-\frac{\mu}{2}(p-1)}\langle \eta\rangle^{-\kappa p}\,d\eta\, d\xi \\
& \lesssim \langle t-r \rangle^{-\kappa-\gamma-\frac{1}{2}+\frac{\mu}{2}(p-1)} \int_{0}^{\frac{t-r}{2}} \frac{\langle \xi\rangle^{-\frac{\mu}{2}(p-1)}}{\sqrt{t-r-\xi}}\, d\xi\lesssim \langle t-r \rangle^{-\kappa-\gamma},
\end{align*} where we used the relation $\langle \xi-\eta\rangle \approx\langle \xi\rangle$ in the second inequality,  \eqref{KuKu95 lemma} in the third one and the same estimate for the $\xi-$integral seen before on $\Omega_1$ in the last one.

It remains to check $J_{\gamma,3}(t,r) $ in order to show \eqref{Jgamma est n even}. Being $\big\langle t-r+\tfrac{\eta-\xi}{2}\big\rangle \geq \langle t-r-\xi\rangle $ and  $\big\langle t-r+\tfrac{\eta-\xi}{2}\big\rangle \geq \big\langle \tfrac{\xi+\eta}{2}\big\rangle \gtrsim \langle \xi+\eta \rangle $ valid on $\Omega_3$, then, 
\begin{align*}
\big\langle t-r+\tfrac{\eta-\xi}{2}\big\rangle^{-q-\frac{1}{2}-\gamma} &\lesssim \langle t-r-\xi\rangle^{-q+\frac{1}{2}}  \langle \xi+\eta \rangle ^{-1-\gamma} \qquad \mbox{if} \quad q\geq \tfrac{1}{2}, \\
\big\langle t-r+\tfrac{\eta-\xi}{2}\big\rangle^{-q-\frac{1}{2}-\gamma} &\lesssim \langle t-r-\xi\rangle^{-q} \langle \xi+\eta \rangle ^{-\frac{1}{2}-\gamma} \quad\qquad \mbox{if} \quad q\in [0,\tfrac{1}{2}),\\
\big\langle t-r+\tfrac{\eta-\xi}{2}\big\rangle^{-q-\frac{1}{2}-\gamma} &\lesssim \langle t-r-\xi\rangle^{-q-\frac{1}{2}} \langle \xi+\eta \rangle ^{-\gamma} \quad\qquad \mbox{if} \quad q\in [-\tfrac{1}{2},0].
\end{align*}
Moreover, $\langle \xi -\eta\rangle \approx \langle \xi \rangle \approx \langle \eta \rangle$ and $\langle\xi \rangle \lesssim \langle t-r\rangle^{-\gamma} \langle \xi\rangle^{\gamma-\frac{p}{2}}$ on $\Omega_3$, so, we have
\begin{align*}
J_{\gamma,3}(t,r) & \lesssim \langle t-r\rangle^{-\gamma} \int_{\frac{t-r}{2}}^{t-r} \langle \xi\rangle^{-\frac{\mu}{2}(p-1)-\kappa p+\gamma-\frac{p}{2}} \frac{\langle t-r-\xi\rangle^{-q+\theta(q)}}{\sqrt{t-r-\xi}} \\ & \quad\qquad \qquad \qquad\qquad \times\int_{-\xi}^{-\frac{\xi}{2}} \langle \xi+\eta \rangle^{\frac{p}{2}-\gamma -\frac{1}{2}-\theta(q)} \, d\eta \, d\xi \\
&\lesssim \langle t-r\rangle^{-\gamma} \int_{\frac{t-r}{2}}^{t-r} \langle \xi\rangle^{-\frac{\mu}{2}(p-1)-\kappa p+\frac{1}{2}-\theta(q)} \frac{\langle t-r-\xi\rangle^{-q+\theta(q)}}{\sqrt{t-r-\xi}}  \, d\xi \\
&=\langle t-r\rangle^{-\gamma} \int^{-\frac{t-r}{2}}_{-(t-r)} \langle \xi\rangle^{-\frac{\mu}{2}(p-1)-\kappa p+\frac{1}{2}-\theta(q)} \frac{\langle t-r+\xi\rangle^{-q+\theta(q)}}{\sqrt{t-r+\xi}}  \, d\xi,
\end{align*} where $\theta(q)=\tfrac{1}{2}$ if $q\geq \tfrac{1}{2}$, $\theta(q)=0$ if $0\leq q<\tfrac{1}{2}$ and $\theta(q)=-\frac{1}{2}$ if $-\tfrac{1}{2}\leq q<0$ and in the second inequality we used $\tfrac{p}{2}-\gamma -\tfrac{1}{2}-\theta(q)>-1$.
Thanks to Lemmas \ref{Lemma 2nd integral n even, square root}, \ref{Lemma 2nd integral n even, square root q<1/2} and \ref{Lemma 2nd integral n even, square root q<0}, we have
$J_{\gamma,3}(t,r)  \lesssim \langle t-r\rangle^{-\kappa-\gamma}$.
Hence, we proved \eqref{Jgamma est n even}.

 Now, we deal with $P_\gamma(t,r)$. Also in this case we work with $t>r$. Then, since $ \big\langle\tau +\frac{\lambda_-}{2} \big\rangle \gtrsim \langle \lambda_-\rangle$ and  $ \big\langle\tau +\frac{\lambda_-}{2} \big\rangle \gtrsim \langle t-r \rangle$ on the domain of integration, we can estimate $ \big\langle\tau +\frac{\lambda_-}{2} \big\rangle^{-\frac{p}{2}} \lesssim \langle \lambda_-\rangle^{-\frac{p}{2}+\gamma} \langle t-r \rangle^{-\gamma}$. Then,
\begin{align*}
P_\gamma(t,r) &= \int_0^{t-r} \langle \tau \rangle^{-\frac{\mu}{2}(p-1)}\langle \lambda_-\rangle^{-q+\frac{p}{2}-1-\gamma}    \big\langle \tau +\tfrac{\lambda_-}{2} \big\rangle^{-\frac{p}{2}}   \big\langle \tau -\tfrac{\lambda_-}{2}\big\rangle^{-\kappa p} \, d\tau \\
 &\lesssim   \langle t-r \rangle^{-\gamma}\int_0^{t-r} \langle \tau \rangle^{-\frac{\mu}{2}(p-1)}\langle t-r-\tau\rangle^{-q-1}    \langle t-r-3\tau\rangle^{-\kappa p} \, d\tau.
\end{align*}
Performing the change of variables $x=t-r-3\tau$ and using Lemma \ref{Lemma 3rd integral n even, 3 factors}, we find
\begin{align*}
\langle t-r \rangle^{\gamma} P_\gamma(t,r) &\lesssim \int_{-2(t-r)}^{t-r} \langle t-r-x\rangle^{-\frac{\mu}{2}(p-1)}\langle 2(t-r)+x \rangle^{-q-1}    \langle x\rangle^{-\kappa p} \, dx \lesssim   \langle t-r \rangle^{-\kappa}.
\end{align*}

Finally, we consider $Q_\gamma(t,r)$.
Since $\langle \tau-\lambda_-\rangle \gtrsim \langle t-r\rangle ,\langle \lambda_-\rangle $ on the domain of integration, then, it holds $\langle \tau-\lambda_-\rangle^{-\frac{p}{2}} \lesssim \langle t-r\rangle^{-\gamma} \langle \lambda_-\rangle^{-\frac{p}{2}+\gamma}$, which implies

\begin{align*}
Q_\gamma(t,r) &= \langle t-r\rangle^{-\kappa p}  \int_{(t-r)_+}^t \langle \tau \rangle^{-\frac{\mu}{2}(p-1)}\langle \lambda_-\rangle^{-q+\frac{p}{2}-1-\gamma}   \langle \tau -\lambda_-\rangle^{-\frac{p}{2}} \, d\tau \\
&\lesssim  \langle t-r\rangle^{-\kappa p-\gamma}  \int_{(t-r)_+}^t \langle \tau \rangle^{-\frac{\mu}{2}(p-1)}\langle \lambda_-\rangle^{-q-1}   \, d\tau  \lesssim \langle t-r\rangle^{-\kappa p-\gamma} \lesssim \langle t-r\rangle^{-\kappa -\gamma}.
\end{align*}
where in the second inequality we may use \eqref{KuKu95 lemma} to estimate the integral by a constant, due to \eqref{upper bound for kappa n even} and $q>-1$.
Thus, we proved also \eqref{Qgamma est n even}. This concludes the proof.
\end{proof}

\section{Proof of Theorem \ref{thm GER n even semi} } \label{Section proof of the main thm}

In this Section we prove Theorem \ref{thm GER n even semi}, using the estimates from Section \ref{Section preliminaries}.

\begin{prop} \label{Prop ||Lv|| n even} 
Let us consider $p,\kappa,q$ satisfying \eqref{p>p0(n+mu1) prel n even}, \eqref{upper bound for kappa n even}, \eqref{lower bound for kappa n even}, \eqref{q condition n even} and \eqref{p<p Fuj (mu1)}. Let $v\in X_\kappa $ and $\nu\in \mathbb{R}$. Then, the following estimates are satisfied for any $t\geq 0,r>0$
\begin{align}
|Lv(t,r)| &\lesssim N_0^\nu (|v|^p)\,  r^{-m} \phi_\kappa(t,r)  & \mbox{if} \quad \nu>-2 , \label{1st cond Lv n even} \\
|Lv(t,r)| &\lesssim \widetilde{N}_1^\nu (|v|^p) \, r^{1-m} \langle t-r\rangle^{-\kappa-\frac{1}{2}}  & \mbox{if} \quad \nu>-1 , \label{2nd cond Lv n even} \\
|\partial_r Lv(t,r)| &\lesssim \widetilde{N}_1^\nu (|v|^p)\,  r^{-m} \phi_\kappa(t,r)  & \mbox{if} \quad \nu>-1 , \label{3rd cond Lv der n even}
\end{align}
where $\phi_\kappa(t,r)$ is defined by \eqref{def phi kappa} and $\widetilde{N}_1^\nu (|v|^p)=N_0^\nu (|v|^p)+N_1^\nu (|v|^p)$, being $N_0^\nu,N_1^\nu$ defined by \eqref{def Nj nu v^p}.
In particular, it holds
\begin{align}\label{||Lv|| n even estimate}
\|Lv\|_{X_\kappa} \lesssim \widetilde{N}_1^\nu (|v|^p) \qquad \mbox{if} \quad \nu>-1.
\end{align}
\end{prop}

\begin{rmk}\label{rmk lispscitz conditions}
Let $v,\bar{v}\in X_\kappa$. If we replace $N_j^\nu(|v|^p)$ by $N_j^\nu(|v|^p-|\bar{v}|^p)$, then, we obtain for $Lv-L\bar{v}$ the estimates which correspond to \eqref{1st cond Lv n even}, \eqref{2nd cond Lv n even} and \eqref{3rd cond Lv der n even}. 
\end{rmk}

We anticipate to the proof of Proposition \ref{Prop ||Lv|| n even} some lemmas.

\begin{lemma} \label{Lemma W1,3,5}
Let us consider $p,\kappa,q$ satisfying \eqref{p>p0(n+mu1) prel n even}, \eqref{upper bound for kappa n even}, \eqref{lower bound for kappa n even}, \eqref{q condition n even} and \eqref{p<p Fuj (mu1)} and let $\gamma$ be $0$ or $\frac{1}{2}$. Let $v\in X_\kappa$ and let $W_1,W_3,W_5$ be as in \eqref{def W1+W2}, \eqref{def W3+W4} and \eqref{def W5+W6}. Then, the following estimates are valid for any $t\geq 0, r>0$:
\begin{align}
&\int_0^t \langle\tau\rangle^{-\frac{\mu}{2}(p-1)}|W_1(t-\tau,r;\tau)|\, d\tau \label{estimate W1}   \\ & \qquad\qquad \qquad\lesssim N_0^\nu (|v|^p)\, r^{m+\gamma-\frac{1}{2}}\Big(I_\gamma(t,r)+\langle t-r\rangle^{-(\kappa+\frac{1}{2})p}\Big),   & \mbox{if} \quad \nu >-2, \notag \\
&\int_0^t \langle\tau\rangle^{-\frac{\mu}{2}(p-1)}|W_3(t-\tau,r;\tau)|\, d\tau \label{estimate W3} \\ & \qquad\qquad \qquad \lesssim N_1^\nu (|v|^p) \,r^{m+1}\Big(I_\frac{1}{2}(t,r)+\langle t-r\rangle^{-(\kappa+\frac{1}{2})p}\Big),  & \mbox{if} \quad \nu >-1,  \notag\\
&\int_0^t \langle\tau\rangle^{-\frac{\mu}{2}(p-1)}|W_5(t-\tau,r;\tau)|\, d\tau \label{estimate W5} \\ & \qquad\qquad \qquad \lesssim N_1^\nu (|v|^p) \, r^{m+\gamma-\frac{1}{2}}\Big(I_\gamma(t,r)+\langle t-r\rangle^{-(\kappa+\frac{1}{2})p}\Big),  & \mbox{if} \quad \nu >-1, \notag
\end{align}  where $I_\gamma(t,r)$ is given by \eqref{def Igamma n even}.
\end{lemma}

\begin{proof}
%
We will follow the proof of Lemma 6.3 in \cite{KuKu96}.
We begin with the estimate for the integral that involves $W_1$. Since \eqref{def Nj nu v^p} and  \eqref{estimate K m} imply for $j=0,1$
\begin{align}
\big|\partial^j_\lambda \big(\lambda^{2m} |v(\tau,\lambda)|^p\big)\big| \lesssim \lambda^{m+\nu-j}\langle \lambda \rangle^{-q+\frac{p}{2}-\frac{3}{2}-\nu+j}\phi_\kappa(\tau,\lambda)^{p} N_j^\nu(|v|^p) \label{decay rate Nj nu(v^p)}
\end{align}
and \begin{align*}
|K_m(\lambda,t-\tau,r)| &\lesssim r^{m+\gamma-\frac{1}{2}} \lambda^{-m-\gamma} (\lambda-\lambda_-)^{-\frac{1}{2}} \qquad \mbox{for} \quad |\lambda_-|<\lambda<\lambda_+ ,
\end{align*} respectively, by using the representation formula
\begin{align*}
\int_0^t \langle\tau\rangle^{-\frac{\mu}{2}(p-1)}W_1(t-\tau,r;\tau)\, d\tau = \!\int_0^t \!\langle\tau\rangle^{-\frac{\mu}{2}(p-1)}\!\int_{|\lambda_-|}^{\lambda_+} \!\lambda^{2m+1} |v(\tau,\lambda)|^p\, K_m(\lambda,t-\tau,r)\,d\lambda  d\tau,
\end{align*}
we get
\begin{align*}
\int_0^t \langle &\tau\rangle^{-\frac{\mu}{2}(p-1)}|W_1(t-\tau,r;\tau)|\, d\tau \\ & \lesssim N_0^\nu(|v|^p)\,r^{m+\gamma-\frac{1}{2}}\int_0^t \langle\tau\rangle^{-\frac{\mu}{2}(p-1)}\int_{|\lambda_-|}^{\lambda_+} \lambda^{\nu-\gamma+1} \langle \lambda\rangle^{-q+\frac{p}{2}-\frac{3}{2}-\nu}\,\frac{\phi_\kappa(\tau,\lambda)^{p}}{\sqrt{\lambda-\lambda_-}}\,d\lambda \, d\tau, \\
 & \lesssim N_0^\nu(|v|^p)\,r^{m+\gamma-\frac{1}{2}}\Big( I_\gamma(t,r) 
+\int_0^t \langle\tau\rangle^{-\frac{\mu}{2}(p-1)}\int_{|\lambda_-|}^{\min(1,\lambda_+)} \lambda^{\nu-\gamma+1}\,\frac{\phi_\kappa(\tau,\lambda)^{p}}{\sqrt{\lambda-\lambda_-}}\,d\lambda \, d\tau\Big),
\end{align*}
where in the last inequality we used $\lambda\approx \langle\lambda\rangle$ for $\lambda\geq 1$ and $\langle \lambda\rangle^{-q+\frac{p}{2}-\frac{3}{2}-\nu}\approx 1$ for $\lambda\in [0,1]$ and $I_\gamma(t,r)$ is defined by \eqref{def Igamma n even}. In order to show \eqref{estimate W1}, it remains to prove that the second integral in the last line of the previous chain of inequalities can be estimated by $\langle t-r\rangle^{-(\kappa+\frac{1}{2})p}$. 

First of all, $\langle \tau+\lambda \rangle \geq  \langle \tau \rangle$, since $\tau$ and $\lambda$ are nonnegative. Besides, $|\lambda|\leq 1$ implies $\langle\tau -\lambda \rangle \gtrsim \langle \tau \rangle.$ Consequently, $\phi_\kappa(\tau,\lambda)^p \lesssim \langle t-r\rangle^{-(\kappa+\frac{1}{2})p}$ on the domain of integration.
Hence, applying Fubini's theorem, since $\langle\tau\rangle^{-\frac{\mu}{2}(p-1)}\lesssim 1$, we get
\begin{align*}
& \int_0^t \langle\tau\rangle^{-\frac{\mu}{2}(p-1)}\int_{|\lambda_|}^{\min(1,\lambda_+)} \!\lambda^{\nu-\gamma+1}\frac{\phi_\kappa(\tau,\lambda)^{p}}{\sqrt{\lambda-\lambda_-}}\,d\lambda \, d\tau \\
& \qquad\qquad\qquad\lesssim \langle t-r\rangle^{-(\kappa+\frac{1}{2})p} \int_0^1  \lambda^{\nu-\gamma+1} \int_{t-r-\lambda}^{t-r-\lambda}\frac{d\tau}{\sqrt{\lambda-\lambda_-}} \, d\lambda \\
& \qquad\qquad\qquad\lesssim \langle t-r\rangle^{-(\kappa+\frac{1}{2})p} \int_0^1  \lambda^{\nu-\gamma+\frac{3}{2}}\, d\lambda \lesssim \langle t-r\rangle^{-(\kappa+\frac{1}{2})p},
\end{align*} where in the last step we used $\nu-\gamma+\tfrac{3}{2}\geq \nu+1 >-1$ for $\nu>-2$.

The proofs of \eqref{estimate W3} and \eqref{estimate W5} are analogous. Indeed, using \eqref{decay rate Nj nu(v^p)} for $j=1$, the representation formulas
\begin{align*}
W_3(t-\tau,r;\tau) & = \int_{|\lambda_-|}^{\lambda_+} \partial_\lambda(\lambda^{2m}|v(\tau,\lambda)|^p)\, K_{m-1}(\lambda,t-\tau,r) \, d\lambda  , \\
W_5(t-\tau,r;\tau) & = \int_{|\lambda_-|}^{\lambda_+} \partial_\lambda(\lambda^{2m}|v(\tau,\lambda)|^p)\, \partial_r K_{m-1}(\lambda,t-\tau,r) \, d\lambda  , 
\end{align*}  and 
\begin{align}
| K_{m-1}(\lambda,t-\tau,r)| &\lesssim r^{m+1} \lambda^{-m-\frac{1}{2}} (\lambda-\lambda_-)^{-\frac{1}{2}} & \mbox{for} \quad |\lambda_-|<\lambda<\lambda_+, \label{estimate K m-1 tau} \\
|\partial_r K_{m-1}(\lambda,t-\tau,r)| &\lesssim r^{m+\gamma-\frac{1}{2}} \lambda^{-m-\gamma+1} (\lambda-\lambda_-)^{-\frac{1}{2}} & \mbox{for} \quad |\lambda_-|<\lambda<\lambda_+, \label{estimate K m-1 tau der}
\end{align}
where the last two inequalities are derived by \eqref{estimate K m-1}, then, we can follow step by step the previous computations. In the end, the only difference is that we lose one order in the power for $\lambda$ in the second integral, so, we have to require in this case $\nu>-1$ instead of $\nu>-2$. Hence, the proof is complete.
%
\end{proof}

\begin{lemma} \label{Lemma W2,4,6 on [0,t-r]}
Let us consider $p,\kappa,q$ satisfying \eqref{p>p0(n+mu1) prel n even}, \eqref{upper bound for kappa n even}, \eqref{lower bound for kappa n even}, \eqref{q condition n even} and \eqref{p<p Fuj (mu1)} and let $\gamma$ be $0$ or $\frac{1}{2}$. Let $v\in X_\kappa$ and let $W_2,W_4,W_6$ be as in \eqref{def W1+W2}, \eqref{def W3+W4} and \eqref{def W5+W6}. Then, the following estimates are valid for any $t\geq 0, r>0$ such that $t>r$:
\begin{align}
&\int_0^{t-r} \langle\tau\rangle^{-\frac{\mu}{2}(p-1)}|W_2(t-\tau,r;\tau)|\, d\tau \label{estimate W2 on [0,t-r]} \\ & \quad\qquad \lesssim N_0^\nu (|v|^p)\, r^{m+\gamma-\frac{1}{2}}\Big(J_\gamma(t,r)+\langle t-r\rangle^{-(\kappa+\frac{1}{2})p}\Big),  & \mbox{if} \quad \nu >-2, \notag \\
&\int_0^{t-r} \langle\tau\rangle^{-\frac{\mu}{2}(p-1)}|W_4(t-\tau,r;\tau)|\, d\tau \label{estimate W4 on [0,t-r]}  \\ & \quad\qquad\lesssim \widetilde{N}_1^\nu (|v|^p)\, r^{m+1}\Big(J_\frac{1}{2}(t,r)+P_\frac{1}{2}(t,r)+\langle t-r\rangle^{-(\kappa+\frac{1}{2})p}\Big), & \mbox{if} \quad \nu >-1, \notag \\
&\int_0^{t-r} \langle\tau\rangle^{-\frac{\mu}{2}(p-1)}|W_6(t-\tau,r;\tau)|\, d\tau \label{estimate W6 on [0,t-r]} \\ & \quad\qquad \lesssim \widetilde{N}_1^\nu (|v|^p) \, r^{m+\gamma-\frac{1}{2}}\Big(J_\gamma(t,r)+P_\gamma(t,r)+\langle t-r\rangle^{-(\kappa+\frac{1}{2})p}\Big), & \mbox{if} \quad \nu >-1, \notag
\end{align} where $J_\gamma(t,r)$ and $P_\gamma(t,r)$ are given by \eqref{def Jgamma n even} and \eqref{def Pgamma n even}, respectively.
\end{lemma}

\begin{proof}
In this case we will modify the proof of Lemma 6.4 in \cite{KuKu96}.
Let us start with the proof of \eqref{estimate W2 on [0,t-r]}. Using the representation formula 
\begin{align*}
\int_0^{t-r} \langle\tau\rangle^{-\frac{\mu}{2}(p-1)}& W_2(t-\tau,r;\tau)\, d\tau \\ &= \int_0^{t-r} \langle\tau\rangle^{-\frac{\mu}{2}(p-1)}\int^{\lambda_-}_0 \lambda^{2m+1} |v(\tau,\lambda)|^p\, \widetilde{K}_m(\lambda,t-\tau,r)\,d\lambda \, d\tau,
\end{align*}
the condition \eqref{decay rate Nj nu(v^p)} for $j=0$ and 
\begin{align*}
|\widetilde{K}_m(\lambda,t-\tau,r)| &\lesssim r^{m+\gamma-\frac{1}{2}} \lambda_-^{-m-\gamma} (\lambda_- -\lambda)^{-\frac{1}{2}} \qquad \mbox{for} \quad 0<\lambda<\lambda_-, 
\end{align*} where the previous inequality follows from \eqref{estimate Ktilde m}, we obtain
\begin{align}
&\int_0^{t-r} \langle\tau\rangle^{-\frac{\mu}{2}(p-1)}|W_2(t-\tau,r;\tau)|\, d\tau \lesssim N_0^\nu(|v|^p)\,r^{m+\gamma-\frac{1}{2}}  \label{intermediate estimate W2}\\ &  \qquad \qquad \times \int_0^{t-r} \langle\tau\rangle^{-\frac{\mu}{2}(p-1)}\lambda_-^{-m-\gamma}  \int^{\lambda_-}_0 \lambda^{m+1+\nu} \langle \lambda\rangle^{-q+\frac{p}{2}-\frac{3}{2}-\nu}\,\frac{\phi_\kappa(\tau,\lambda)^{p}}{\sqrt{\lambda_- -\lambda}}\,d\lambda \, d\tau. \notag 
\end{align}
The next step is to split the $\tau-$integral on two intervals divided by $(t-r-1)_+$. On $[0,(t-r-1)_+]$, we have $\lambda_-\geq 1$ and, then, $\lambda_-\approx\langle \lambda_-\rangle$. Therefore,
\begin{align*}
&\int_0^{(t-r-1)_+}  \langle\tau\rangle^{-\frac{\mu}{2}(p-1)}\lambda_-^{-m-\gamma}\int^{\lambda_-}_0 \lambda^{m+1+\nu} \langle \lambda\rangle^{-q+\frac{p}{2}-\frac{3}{2}-\nu}\,\frac{\phi_\kappa(\tau,\lambda)^{p}}{\sqrt{\lambda_- -\lambda}}\,d\lambda \, d\tau\\& \qquad\lesssim \int_0^{(t-r-1)_+} \langle\tau\rangle^{-\frac{\mu}{2}(p-1)}\langle\lambda_-\rangle^{-m-\gamma}\int^{\lambda_-}_0 \lambda^{m+1+\nu} \langle \lambda\rangle^{-q+\frac{p}{2}-\frac{3}{2}-\nu}\,\frac{\phi_\kappa(\tau,\lambda)^{p}}{\sqrt{\lambda_- -\lambda}}\,d\lambda \, d\tau \\
& \qquad \lesssim \int_0^{(t-r-1)_+} \langle\tau\rangle^{-\frac{\mu}{2}(p-1)}\langle\lambda_-\rangle^{-m-\gamma}\int^{\lambda_-}_0  \langle \lambda\rangle^{-q+\frac{p}{2}-\frac{1}{2}+m}\,\frac{\phi_\kappa(\tau,\lambda)^{p}}{\sqrt{\lambda_- -\lambda}}\,d\lambda \, d\tau\\
& \qquad\lesssim \int_0^{(t-r-1)_+} \langle\tau\rangle^{-\frac{\mu}{2}(p-1)}\langle\lambda_-\rangle^{-\gamma-q-\frac{1}{2}}\int^{\lambda_-}_0  \frac{\langle \lambda\rangle^{\frac{p}{2}} \phi_\kappa(\tau,\lambda)^{p}}{\sqrt{\lambda_- -\lambda}}\,d\lambda \, d\tau \lesssim J_\gamma(t,r),
\end{align*} where in the second inequality we employed the condition $m+\nu+1>0$ (we are assuming $\nu>-2$) and  in the third inequality the upper bound for $q$ in \eqref{q condition n even} is used to get $$\langle\lambda_-\rangle^{-m-\gamma} \langle \lambda\rangle^{-q+\frac{p}{2}-\frac{1}{2}+m}\leq \langle\lambda_-\rangle^{-\gamma-q-\frac{1}{2}} \langle \lambda\rangle^{\frac{p}{2}} \qquad \mbox{for} \quad \lambda \in [0, \lambda_-].$$ 
On the other hand, using Fubini's theorem, on $[(t-r-1)_+, t-r]$, we find
\begin{align*}
\int_{(t-r-1)_+}^{t-r}&  \langle\tau\rangle^{-\frac{\mu}{2}(p-1)}\lambda_-^{-m-\gamma}\int^{\lambda_-}_0 \lambda^{m+1+\nu} \langle \lambda\rangle^{-q+\frac{p}{2}-\frac{3}{2}-\nu}\,\frac{\phi_\kappa(\tau,\lambda)^{p}}{\sqrt{\lambda_- -\lambda}}\,d\lambda \, d\tau \\
& = \int_0^1  \lambda^{m+1+\nu} \langle \lambda\rangle^{-q+\frac{p}{2}-\frac{3}{2}-\nu}\int^{t-r-\lambda}_{(t-r-1)_+} \langle\tau\rangle^{-\frac{\mu}{2}(p-1)}\lambda_-^{-m-\gamma}\,\frac{\phi_\kappa(\tau,\lambda)^{p}}{\sqrt{\lambda_- -\lambda}}\, d\tau\,d\lambda   \\
& \lesssim \int_0^1  \lambda^{m+1+\nu}\int^{t-r-\lambda}_{(t-r-1)_+}\lambda_-^{-m-\gamma}\,\frac{\phi_\kappa(\tau,\lambda)^{p}}{\sqrt{\lambda_- -\lambda}}\, d\tau\,d\lambda \\
& \lesssim \langle t-r\rangle^{-(\kappa+\frac{1}{2})p} \int_0^1  \lambda^{m+1+\nu}\int^{t-r-\lambda}_{(t-r-1)_+}\frac{\lambda_-^{-m-\gamma}}{\sqrt{\lambda_- -\lambda}}\, d\tau\,d\lambda,
\end{align*} where in the last inequality we used the estimate
\begin{align}\label{estimate phi^p}
\phi_\kappa(\tau,\lambda)^{p}\lesssim  \langle t-r\rangle^{-(\kappa+\frac{1}{2})p} \qquad \mbox{for} \quad \tau\in [(t-r-1)_+,t-r] \quad \mbox{and} \quad \lambda \in [0,\lambda_-].
\end{align}
Indeed, trivially $\phi_\kappa(\tau,\lambda)^{p} \leq \langle \tau -\lambda\rangle^{-(\kappa+\frac{1}{2})p}$. Moreover, if $t-r>2$, then, $$|\tau-\lambda| \geq \tau -\lambda \geq  t-r-1 -\lambda \geq t-r-2 $$ implies $\langle \tau-\lambda \rangle \gtrsim \langle t-r\rangle$ and, in turn, the desired inequality. On the other hand, for $0<t-r\leq 2$, we have immediately $\phi_\kappa(\tau,\lambda)^{p}\lesssim \langle t-r\rangle^{-(\kappa+\frac{1}{2})p}$, being $\langle t-r\rangle\approx 1$. Let $\varepsilon>0$ be such that $\varepsilon<\min(\tfrac{1}{2},\nu+2)$. For $0\leq \lambda \leq \lambda_-$ we obtain
$$\lambda_-^{-m-\gamma} \leq (\lambda_- - \lambda)^{-\frac{1}{2}+\varepsilon}\lambda^{-m-\gamma+\frac{1}{2}-\varepsilon},$$ 
 due to $\varepsilon<\tfrac{1}{2}$ and  $-m-\gamma+\tfrac{1}{2}-\varepsilon<0$. Hence,
\begin{align}
& \int_0^1  \lambda^{m+1+\nu}\int^{t-r-\lambda}_{(t-r-1)_+}\frac{\lambda_-^{-m-\gamma}}{\sqrt{\lambda_- -\lambda}}\, d\tau\,d\lambda \label{intermediate estimate W2 n2}\\ & \qquad\qquad\lesssim \int_0^1  \lambda^{\nu-\gamma+\frac{3}{2}-\varepsilon}\int^{t-r-\lambda}_{t-r-1}(\lambda_- - \lambda)^{-1+\varepsilon}\, d\tau\,d\lambda 
  \lesssim \int_0^1  \lambda^{\nu-\gamma+\frac{3}{2}-\varepsilon}\,d\lambda \lesssim 1, \notag
\end{align} where in the last inequality the condition $\varepsilon<\nu+2$ implies the boundedness of the integral. Summarizing, if we combine the estimate for the integrals over $[0,(t-r-1)_+]$ and $[(t-r-1)_+,t-r]$, then, it follows \eqref{estimate W2 on [0,t-r]}.

Let us prove now \eqref{estimate W6 on [0,t-r]}. We consider the representation formula 
\begin{align}
\label{repres W6}\int_0^{t-r} \langle\tau &\rangle^{-\frac{\mu}{2}(p-1)} W_6(t-\tau,r;\tau)\, d\tau \\ &= \int_0^{t-r}\langle\tau\rangle^{-\frac{\mu}{2}(p-1)}\int^{\lambda_-}_0 \!\partial_\lambda(\lambda^{2m} |v(\tau,\lambda)|^p)\, \partial_r\widetilde{K}_{m-1}(\lambda,t-\tau,r)\,d\lambda \, d\tau. \notag
\end{align}
We will split the inner $\lambda-$integral in two parts. We begin with the integral over $[\tfrac{\lambda_-}{2},\lambda_-]$. From \eqref{estimate Ktilde m-1} it follows:
\begin{align*}
|\partial_r \widetilde{K}_{m-1}(\lambda,t-\tau,r)| &\lesssim r^{m+\gamma-\frac{1}{2}} \lambda_-^{-m-\gamma+1} (\lambda_- -\lambda)^{-\frac{1}{2}}  \qquad \mbox{for} \quad \lambda \in (0, \lambda_-).
\end{align*} 
Thus, combining the previous estimate with \eqref{decay rate Nj nu(v^p)} for $j=1$, we find
\begin{align*}
&\int_0^{t-r}  \langle\tau\rangle^{-\frac{\mu}{2}(p-1)}\int^{\lambda_-}_{\lambda_-/2} \big|\partial_\lambda(\lambda^{2m} |v(\tau,\lambda)|^p)\, \partial_r\widetilde{K}_{m-1}(\lambda,t-\tau,r)\big|\,d\lambda \, d\tau \\
 &  \lesssim N_1^\nu(|v|^p)\,r^{m+\gamma-\frac{1}{2}} \! \int_0^{t-r} \!\!\langle\tau\rangle^{-\frac{\mu}{2}(p-1)}\lambda_-^{-m-\gamma+1}\!\int^{\lambda_-}_{\lambda_-/2} \!\lambda^{m+\nu-1} \langle \lambda\rangle^{-q+\frac{p}{2}-\frac{1}{2}-\nu}\,\frac{\phi_\kappa(\tau,\lambda)^{p}}{\sqrt{\lambda_- -\lambda}}\,d\lambda  d\tau.
\end{align*}
In the last integral we consider a further division of the domain of integration, in this case with respect to the $\tau-$integral. On the one hand, it holds
\begin{align*}
 \int_0^{(t-r-1)_+}& \langle\tau\rangle^{-\frac{\mu}{2}(p-1)}\lambda_-^{-m-\gamma+1}\int^{\lambda_-}_{\lambda_-/2} \lambda^{m+\nu-1} \langle \lambda\rangle^{-q+\frac{p}{2}-\frac{1}{2}-\nu}\,\frac{\phi_\kappa(\tau,\lambda)^{p}}{\sqrt{\lambda_- -\lambda}}\,d\lambda \, d\tau \\
& \lesssim  \int_0^{(t-r-1)_+} \langle\tau\rangle^{-\frac{\mu}{2}(p-1)} \langle \lambda_-\rangle^{-q-\frac{1}{2}-\gamma}\int^{\lambda_-}_{\lambda_-/2} \frac{\langle \lambda\rangle^{\frac{p}{2}} \phi_\kappa(\tau,\lambda)^{p}}{\sqrt{\lambda_- -\lambda}}\,d\lambda \, d\tau \lesssim J_\gamma(t,r),
\end{align*} where  we use  $\lambda\approx\lambda_- \approx \langle \lambda_-\rangle$ thanks to $\lambda \in\big[\tfrac{\lambda_-}{2},\lambda_-\big]$ and $\tau\leq t-r-1$.
On the other hand, using Fubini's theorem, for the second part we get
\begin{align*}
\int_{(t-r-1)_+}^{t-r} & \langle\tau\rangle^{-\frac{\mu}{2}(p-1)}\lambda_-^{-m-\gamma+1}\int^{\lambda_-}_{\lambda_-/2} \lambda^{m+\nu-1} \langle \lambda\rangle^{-q+\frac{p}{2}-\frac{1}{2}-\nu}\,\frac{\phi_\kappa(\tau,\lambda)^{p}}{\sqrt{\lambda_- -\lambda}}\,d\lambda \, d\tau \\
& \lesssim \int_0^1 \lambda^{m+\nu} \langle \lambda\rangle^{-q+\frac{p}{2}-\frac{1}{2}-\nu} \int_{(t-r-1)_+}^{t-r-\lambda} \langle\tau\rangle^{-\frac{\mu}{2}(p-1)}\lambda_-^{-m-\gamma} \,\frac{\phi_\kappa(\tau,\lambda)^{p}}{\sqrt{\lambda_- -\lambda}} \, d\tau \,d\lambda \\
& \lesssim \langle t-r\rangle^{-(\kappa+\frac{1}{2})p}\int_0^1 \lambda^{m+\nu}  \int_{(t-r-1)_+}^{t-r-\lambda}  \frac{\lambda_-^{-m-\gamma}}{\sqrt{\lambda_- -\lambda}} \, d\tau \,d\lambda ,
\end{align*} where in the last inequality we used \eqref{estimate phi^p}. Choosing $\varepsilon<\min(\tfrac{1}{2},\nu+1)$, we can repeat exactly the same estimate seen in \eqref{intermediate estimate W2 n2} for the last integral, requiring $\nu>-1$.
Summarizing, we have shown
\begin{align*}
\int_0^{t-r}  \langle\tau\rangle^{-\frac{\mu}{2}(p-1)}\int^{\lambda_-}_{\lambda_-/2} \big|\partial_\lambda(\lambda^{2m} |v(\tau,\lambda)|^p)&\, \partial_r\widetilde{K}_{m-1}(\lambda,t-\tau,r)\big|\,d\lambda \, d\tau \\
 &  \lesssim N_1^\nu(|v|^p)\,r^{m+\gamma-\frac{1}{2}}  \Big(J_\gamma(t,r)+\langle t-r\rangle^{-(\kappa+\frac{1}{2})p}\Big).
\end{align*}

Let us deal with the second term coming from the $\lambda-$integral in \eqref{repres W6}. Integrating by parts, we have
\begin{align*}
\int_0^{\lambda_-/2}  &\partial_\lambda(\lambda^{2m} |v(\tau,\lambda)|^p)\, \partial_r\widetilde{K}_{m-1}(\lambda,t-\tau,r)\,d\lambda  \\ &= \lambda^{2m} |v(\tau,\lambda)|^p \, \partial_r\widetilde{K}_{m-1}(\lambda,t-\tau,r)\Big|_{\lambda= \lambda_-/2}\\ & \quad -\int_0^{\lambda_-/2} \lambda^{2m} |v(\tau,\lambda)|^p\, \partial_\lambda \partial_r\widetilde{K}_{m-1}(\lambda,t-\tau,r)\,d\lambda  \doteq W_{6,1}+W_{6,2}.
\end{align*}
Let us begin with $W_{6,1}$. Using \eqref{decay rate Nj nu(v^p)} and \eqref{estimate Ktilde m-1}, one gets
\begin{align*}
\int_0^{t-r} &\langle\tau\rangle^{-\frac{\mu}{2}(p-1)}|W_{6,1}| \,d\tau \\ 
& \lesssim N_0^\nu(|v|^p)\,r^{m+\gamma-\frac{1}{2}} \int_0^{t-r} \langle\tau\rangle^{-\frac{\mu}{2}(p-1)} \lambda_-^{\nu-\gamma+\frac{1}{2}} \langle \lambda_-\rangle^{-q+\frac{p}{2}-\frac{3}{2}-\nu}  \phi_\kappa(\tau,\tfrac{\lambda_-}{2})^p \, d\tau.
\end{align*}
We split now the $\tau-$integral as usual. On the one hand,
\begin{align*}
\int_0^{(t-r-1)_+} &\langle\tau\rangle^{-\frac{\mu}{2}(p-1)} \lambda_-^{\nu-\gamma+\frac{1}{2}} \langle \lambda_-\rangle^{-q+\frac{p}{2}-\frac{3}{2}-\nu}  \phi_\kappa(\tau,\tfrac{\lambda_-}{2})^p \, d\tau \\ 
& \lesssim \int_0^{(t-r-1)_+} \langle\tau\rangle^{-\frac{\mu}{2}(p-1)} \langle \lambda_-\rangle^{-q+\frac{p}{2}-1-\gamma}  \phi_\kappa(\tau,\tfrac{\lambda_-}{2})^p \, d\tau \leq P_\gamma(t,r).
\end{align*}
On the other hand,  \eqref{estimate phi^p} yields
\begin{align*}
\int_{(t-r-1)_+}^{t-r} \langle\tau\rangle^{-\frac{\mu}{2}(p-1)} & \lambda_-^{\nu-\gamma+\frac{1}{2}} \langle \lambda_-\rangle^{-q+\frac{p}{2}-\frac{3}{2}-\nu}  \phi_\kappa(\tau,\tfrac{\lambda_-}{2})^p \, d\tau  \\ &\lesssim \langle t-r\rangle^{-(\kappa+\frac{1}{2})p} \int_{(t-r-1)_+}^{t-r}  \lambda_-^{\nu-\gamma+\frac{1}{2}}   \, d\tau \lesssim \langle t-r\rangle^{-(\kappa+\frac{1}{2})p} ,
\end{align*} where in the first inequality we also employed $\langle \lambda_-\rangle \approx 1$ and in the second one the assumption $\nu>-1$ is necessary to guarantee the finiteness of the integral.
So, we proved
\begin{align*}
\int_0^{t-r} &\langle\tau\rangle^{-\frac{\mu}{2}(p-1)}|W_{6,1}| \,d\tau \lesssim N_0^\nu(|v|^p)\,r^{m+\gamma-\frac{1}{2}} \Big(P_\gamma(t,r)+\langle t-r\rangle^{-(\kappa+\frac{1}{2})p}\Big).
\end{align*}
We consider now the integral involving $W_{6,2}$. From \eqref{estimate Ktilde j der}, we have
\begin{align*}
|\partial_\lambda \partial_r \widetilde{K}_{m-1}(\lambda,t-\tau,r)|\lesssim r^{m+\gamma-\frac{1}{2}}\lambda_-^{-m+1-\gamma}(\lambda_- -\lambda)^{-\frac{3}{2}}  \qquad \mbox{for} \quad  \lambda \in (0, \lambda_-).
\end{align*}
Combining the previous estimate with \eqref{decay rate Nj nu(v^p)} for $j=0$, we obtain
\begin{align*}
&\int_0^{t-r} \langle\tau\rangle^{-\frac{\mu}{2}(p-1)}|W_{6,2}| \,d\tau \\
& \lesssim N_0^\nu (|v|^p) \, r^{m+\gamma-\frac{1}{2}} \!\int_0^{t-r}\! \langle\tau\rangle^{-\frac{\mu}{2}(p-1)} \!\int_0^{\lambda_-/2} \!\lambda^{m+\nu}\langle \lambda\rangle^{-q+\frac{p}{2}-\frac{3}{2}-\nu} \lambda_-^{-m+1-\gamma} \, \frac{\phi_\kappa(\tau,\lambda)^p}{(\lambda_- -\lambda)^\frac{3}{2}} \, d\lambda\, d\tau \\
& \lesssim N_0^\nu (|v|^p) \, r^{m+\gamma-\frac{1}{2}} \!\int_0^{t-r}\! \langle\tau\rangle^{-\frac{\mu}{2}(p-1)} \!\int_0^{\lambda_-/2} \!\lambda^{m+\nu}\langle \lambda\rangle^{-q+\frac{p}{2}-\frac{3}{2}-\nu} \lambda_-^{-m-\gamma} \, \frac{\phi_\kappa(\tau,\lambda)^p}{\sqrt{\lambda_- -\lambda}} \, d\lambda\, d\tau,
\end{align*} where in the last step the relation $\lambda_- -\lambda \geq \tfrac{\lambda_-}{2}$ is used. The right-hand side of the previous chain of inequality may be estimated exactly as the right-hand side in \eqref{intermediate estimate W2}. The only difference is the power for $\lambda$, so that, in this case we have to require $\nu>-1$ instead of $\nu>-2$ . Therefore, it holds
\begin{align*}
\int_0^{t-r} &\langle\tau\rangle^{-\frac{\mu}{2}(p-1)}|W_{6,2}|\lesssim N_0^\nu(|v|^p)\,r^{m+\gamma-\frac{1}{2}} \Big(J_\gamma(t,r)+\langle t-r\rangle^{-(\kappa+\frac{1}{2})p}\Big).
\end{align*} 
Combining the estimates for the integrals involving $W_{6,1}$ and $W_{6,2}$, it follows \eqref{estimate W6 on [0,t-r]}. 

Finally, \eqref{estimate Ktilde m-1} and \eqref{estimate Ktilde j der} imply for $\gamma=\frac{1}{2}$
\begin{align*}
|\widetilde{K}_{m-1}(\lambda,t-\tau,r)| &\lesssim r^{m+1} \lambda_-^{-m+\frac{1}{2}} (\lambda_- -\lambda)^{-\frac{1}{2}}  & \mbox{for} \quad  \lambda \in (0, \lambda_-), \\
|\partial_\lambda \widetilde{K}_{m-1}(\lambda,t,r)| &\lesssim r^{m+1} \lambda_-^{-m+\frac{1}{2}} (\lambda_- -\lambda)^{-\frac{3}{2}}  & \mbox{for} \quad  \lambda \in (0, \lambda_-).
\end{align*}
Thus, using these estimates and the  representation formula
\begin{align*}
\int_0^{t-r} \langle\tau\rangle^{-\frac{\mu}{2}(p-1)}& W_4(t-\tau,r;\tau)\, d\tau \\ &= \int_0^{t-r}\langle\tau\rangle^{-\frac{\mu}{2}(p-1)}\int^{\lambda_-}_0 \partial_\lambda(\lambda^{2m} |v(\tau,\lambda)|^p)\, \widetilde{K}_{m-1}(\lambda,t-\tau,r)\,d\lambda \, d\tau,
\end{align*}
 one can prove \eqref{estimate W4 on [0,t-r]} exactly as \eqref{estimate W6 on [0,t-r]} has just been proved.
%
%
\end{proof}

\begin{lemma} \label{Lemma W2,4,6 on [t-r,t]}
Let us consider $p,\kappa,q$ satisfying \eqref{p>p0(n+mu1) prel n even}, \eqref{upper bound for kappa n even}, \eqref{lower bound for kappa n even}, \eqref{q condition n even} and \eqref{p<p Fuj (mu1)} and let $\gamma$ be $0$ or $\frac{1}{2}$. Let $v\in X_\kappa$ and let $W_2,W_4,W_6$ be as in \eqref{def W1+W2}, \eqref{def W3+W4} and \eqref{def W5+W6}. Then, the following estimates are valid for any $t\geq 0, r>0$:
\begin{align}
&\int_{(t-r)_+}^t \langle\tau\rangle^{-\frac{\mu}{2}(p-1)}W_2(t-\tau,r;\tau)\, d\tau =0\, \label{estimate W2 on [t-r,t]} \\
&\int_{(t-r)_+}^t  \langle\tau\rangle^{-\frac{\mu}{2}(p-1)}|W_4(t-\tau,r;\tau)|\, d\tau \label{estimate W4 on [t-r,t]} \\ &\qquad \qquad\qquad\lesssim N_0^\nu (|v|^p)\, r^{m+1}\Big(Q_\frac{1}{2}(t,r)+\langle t-r\rangle^{-(\kappa+\frac{1}{2})p}\Big),   & \mbox{if} \quad \nu >-1, \notag  \\
&\int_{(t-r)_+}^t \langle\tau\rangle^{-\frac{\mu}{2}(p-1)} |W_6(t-\tau,r;\tau)|\, d\tau \label{estimate W6 on [t-r,t]} \\ & \qquad\qquad\qquad \lesssim N_0^\nu (|v|^p) \, r^{m+\gamma-\frac{1}{2}}\Big(Q_\gamma(t,r)+\langle t-r\rangle^{-(\kappa+\frac{1}{2})p}\Big),  & \mbox{if} \quad \nu >-1, \notag
\end{align} where $Q_\gamma(t,r)$ is given by \eqref{def Qgamma n even}.
\end{lemma}

\begin{proof}
We will adapt the proof of Lemma 6.5 in \cite{KuKu96} to our case.
From \eqref{def w2} we get immediately \eqref{estimate W2 on [t-r,t]}, being $t-\tau-r\leq 0$. 

Now we prove \eqref{estimate W6 on [t-r,t]}. 
Using \eqref{estimate K m-1 tau der}, \eqref{decay rate Nj nu(v^p)} for $j=0$ and the representation formula
\begin{align*}
\int_{(t-r)_+}^t \langle\tau\rangle^{-\frac{\mu}{2}(p-1)}& W_6(t-\tau,r;\tau)\, d\tau \\&= \int_{(t-r)_+}^t \langle\tau\rangle^{-\frac{\mu}{2}(p-1)} \Big(\lambda^{2m} |v(\tau,\lambda)|^p \partial_r K_{m-1}(\lambda,t-\tau,r)\Big)\Big|_{\lambda=-\lambda_-} \, d\tau,
\end{align*}
we find
\begin{align*}
\int_{(t-r)_+}^t &\langle\tau\rangle^{-\frac{\mu}{2}(p-1)} |W_6(t-\tau,r;\tau)|\, d\tau \\
 &\lesssim N_0^\nu (|v|^p) \, r^{m+\gamma-\frac{1}{2}}\int_{(t-r)_+}^t \langle\tau\rangle^{-\frac{\mu}{2}(p-1)} |\lambda_-|^{\nu+\frac{1}{2}-\gamma}\langle \lambda_- \rangle^{-q+\frac{p}{2}-\frac{3}{2}-\nu}\phi_\kappa(\tau,-\lambda_-)^p \, d\tau.
\end{align*}

We divide the integral in two parts. Firstly,
\begin{align*}
&\int_{(t-r)_+}^{(t-r+1)_+}  \langle\tau\rangle^{-\frac{\mu}{2}(p-1)} |\lambda_-|^{\nu+\frac{1}{2}-\gamma}\langle \lambda_- \rangle^{-q+\frac{p}{2}-\frac{3}{2}-\nu}\phi_\kappa(\tau,-\lambda_-)^p \, d\tau \\
&  \quad \quad\lesssim \int_{(t-r)_+}^{(t-r+1)_+}  |\lambda_-|^{\nu+\frac{1}{2}-\gamma}\phi_\kappa(\tau,-\lambda_-)^p \, d\tau \lesssim \langle t-r\rangle^{-(\kappa+\frac{1}{2})p} \int_{(t-r)_+}^{(t-r+1)_+}  |\lambda_-|^{\nu+\frac{1}{2}-\gamma}\, d\tau
\end{align*} here in the first inequality $\langle \lambda_-\rangle\approx 1$ is used, while in the second inequality we used $$\phi_\kappa(\tau,-\lambda_-)^p \lesssim \langle t-r\rangle^{-(\kappa+\frac{1}{2})p}.$$ Indeed, $ \tau+\lambda_-=t-r$ and $\tau-\lambda_-=(r-t)+2\tau \geq |t-r|$ for $\tau \geq (t-r)_+$ imply the previous inequality. Since $\nu>-1$, then,
\begin{align*}
 \int_{(t-r)_+}^{(t-r+1)_+}  |\lambda_-|^{\nu+\frac{1}{2}-\gamma}\, d\tau \leq  \int_{t-r}^{t-r+1}  |\lambda_-|^{\nu+\frac{1}{2}-\gamma}\, d\tau =\int_{0}^{1}  \tau^{\nu+\frac{1}{2}-\gamma}\, d\tau \lesssim 1.
\end{align*}
So, we proved,
\begin{align*}
\int_{(t-r)_+}^{(t-r+1)_+}  \langle\tau\rangle^{-\frac{\mu}{2}(p-1)}& W_6(t-\tau,r;\tau)\, d\tau  \lesssim  N_0^\nu (|v|^p) \, r^{m+\gamma-\frac{1}{2}}  \langle t-r\rangle^{-(\kappa+\frac{1}{2})p}.
\end{align*}
Finally,
\begin{align*}
&\int_{(t-r+1)_+}^t  \langle\tau\rangle^{-\frac{\mu}{2}(p-1)} W_6(t-\tau,r;\tau)\, d\tau  \\ &\quad\quad \lesssim N_0^\nu (|v|^p) \, r^{m+\gamma-\frac{1}{2}}\int_{(t-r+1)_+}^t \langle\tau\rangle^{-\frac{\mu}{2}(p-1)} \langle \lambda_- \rangle^{-q+\frac{p}{2}-1-\gamma}\phi_\kappa(\tau,-\lambda_-)^p \, d\tau \leq Q_\gamma(t,r),
\end{align*} being $|\lambda_-|\approx \langle \lambda_-\rangle$ on the domain of integration. Also, we showed \eqref{estimate W6 on [t-r,t]}. 

Analogously, by the representation formula
\begin{align*}
\int_{(t-r)_+}^t \langle\tau\rangle^{-\frac{\mu}{2}(p-1)}& W_4(t-\tau,r;\tau)\, d\tau \\ &= \int_{(t-r)_+}^t \langle\tau\rangle^{-\frac{\mu}{2}(p-1)} \Big(\lambda^{2m} |v(\tau,\lambda)|^p  K_{m-1}(\lambda,t-\tau,r)\Big)\Big|_{\lambda=-\lambda_-} \, d\tau
\end{align*} it is possible to show \eqref{estimate W4 on [t-r,t]} exactly as we have proved \eqref{estimate W6 on [t-r,t]}. In particular, one has to employ the inequality \eqref{estimate K m-1 tau}. This concludes the proof.
\end{proof}

\begin{proof}[Proof of Proposition \ref{Prop ||Lv|| n even}] In order to represent $Lv$ and $\partial_r Lv$, we will use \eqref{def W1+W2}, \eqref{def W3+W4} and \eqref{def W5+W6}. Since $-(\kappa+\frac{1}{2})p\leq -\kappa-\gamma$, combining Lemmas \ref{Lemma W1,3,5}, \ref{Lemma W2,4,6 on [0,t-r]} and \ref{Lemma W2,4,6 on [t-r,t]} and Proposition \ref{Lemma I,J,P,Q gamma=0,1/2},  we get for $\gamma\in\{0,\frac{1}{2}\}$
\begin{align*}
\int_0^t \langle\tau\rangle^{-\frac{\mu}{2}(p-1)} |W_i(t-\tau,r;\tau)| \, d\tau & \lesssim N_0^\nu(|v|^p)\, r^{m+\gamma-\frac{1}{2}} \langle t-r\rangle^{-\kappa-\gamma} & \mbox{for} \,\,\,  \nu>-2 \, ,\, i=1,2,\\
\int_0^t \langle\tau\rangle^{-\frac{\mu}{2}(p-1)} |W_i(t-\tau,r;\tau)| \, d\tau & \lesssim \widetilde{N}_1^\nu(|v|^p)\, r^{m+\gamma-\frac{1}{2}} \langle t-r\rangle^{-\kappa-\gamma} & \mbox{for} \,\,\,  \nu>-1\, , \,  i=5,6 ,
\end{align*} and for  $\nu>-1$,  $i=3,4$.
\begin{align}\label{est int W_3,W_4}
\int_0^t \langle\tau\rangle^{-\frac{\mu}{2}(p-1)} |W_i(t-\tau,r;\tau)| \, d\tau & \lesssim \widetilde{N}_1^\nu(|v|^p) \,r^{m+1} \langle t-r\rangle^{-\kappa-\frac{1}{2}}. 
\end{align}

We note that for $\gamma=0$ or $\gamma=\frac{1}{2}$ it holds 
\begin{align}\label{gamma ineq}
r^{\gamma-\frac{1}{2}} \langle t-r\rangle^{-\gamma}\lesssim \langle t+r\rangle^{-\frac{1}{2}}.
\end{align}
Indeed, for $t\geq 2r>0$ or $r\leq 1$ we have $\langle t+r\rangle\approx \langle t-r\rangle$, so \eqref{gamma ineq} is valid for $\gamma=\frac{1}{2}$. On the other hand, for $0\leq t\leq 2r $ and $r\geq 1$ we have $ r \approx \langle r \rangle \gtrsim \langle t+r\rangle $, thus \eqref{gamma ineq} is satisfied for $\gamma=0$. 
Hence, from the previous first two estimates, we find 
\begin{align}
\int_0^t \langle\tau\rangle^{-\frac{\mu}{2}(p-1)} |W_i(t-\tau,r;\tau)| \, d\tau & \lesssim N_0^\nu(|v|^p)\, r^{m} \phi_\kappa(t,r) & \mbox{for} \,\,\, \nu>-2  \, , \, i=1,2, \label{est int W_1,W_2}\\
\int_0^t \langle\tau\rangle^{-\frac{\mu}{2}(p-1)} |W_i(t-\tau,r;\tau)| \, d\tau & \lesssim \widetilde{N}_1^\nu(|v|^p)\, r^{m} \phi_\kappa(t,r) & \mbox{for} \,\,\, \nu>-1 \, , \, i=5,6. \label{est int W_5,W_6}
\end{align}

Let us prove now \eqref{1st cond Lv n even}. Combining \eqref{def Lv case n even} and \eqref{def W1+W2} and employing \eqref{est int W_1,W_2}, we have for $\nu>-2$
\begin{align*}
| Lv(t,r)| \lesssim r^{-2m} \!\!\int_0^t \!\langle\tau \rangle^{-\frac{\mu}{2}(p-1)} | W_1(t-\tau,r;\tau)\!+ \!W_2(t-\tau,r;\tau)| d\tau \lesssim N_0^\nu(|v|^p) r^{-m} \phi_\kappa(t,r).
\end{align*}

Similarly, using \eqref{def W3+W4} and \eqref{est int W_3,W_4} instead of 
\eqref{def W1+W2} and \eqref{est int W_1,W_2}, respectively, it follows \eqref{2nd cond Lv n even}. Let us show now \eqref{3rd cond Lv der n even}. By using \eqref{def Lv case n even} and \eqref{def W5+W6}, we have
\begin{align*}
 | \partial_r Lv(t,r)|& \approx \Big|\partial_r \Big(r^{-2m}\int_0^t \langle\tau\rangle^{-\frac{\mu}{2}(p-1)}\big(2r^{2m} \Theta(|v(\tau,\cdot)|^p)(t-\tau,r)\big)\, d\tau \Big| 
  \\ & \lesssim |r^{-1}Lv(t,r)| +  \Big| r^{-2m}\int_0^t \langle\tau\rangle^{-\frac{\mu}{2}(p-1)}(W_5(t-\tau,r;\tau)+W_6(t-\tau,r;\tau)\big)\, d\tau \Big|.
\end{align*} We can estimate the second term in the last line with $\widetilde{N}_1^\nu(|v|^p) r^{-m} \phi_\kappa(t,r)$ thanks to \eqref{est int W_5,W_6}. Also, in order to prove \eqref{3rd cond Lv der n even}, it remains to show that $|r^{-1}Lv(t,r)|$ can be controlled by the same quantity as the second term. Let us distinguish two subcases.
When $r\leq 1$, then, since $\langle t+r\rangle\approx \langle t-r\rangle$, by using \eqref{2nd cond Lv n even}, we obtain 
\begin{align*}
|r^{-1}Lv(t,r)| \lesssim r^{-1}  \widetilde{N}_1^\nu(|v|^p)\, r^{1-m} \langle t-r\rangle^{-(\kappa+\frac{1}{2})}\approx \widetilde{N}_1^\nu(|v|^p)\, r^{-m}\phi_\kappa(t,r).
\end{align*}
Else, for $r\geq 1 $ from \eqref{1st cond Lv n even} we get immediately
\begin{align*}
|r^{-1}Lv(t,r)| \lesssim r^{-1}  N_0^\nu(|v|^p) \,r^{-m} \phi_\kappa(t,r)\lesssim   N_0^\nu(|v|^p)\, r^{-m} \phi_\kappa(t,r).
\end{align*}

Finally, we prove \eqref{||Lv|| n even estimate}. We can rewrite \eqref{3rd cond Lv der n even} as 
\begin{align}\label{3rd cond Lv der n even, rewritten}
r^m \phi_\kappa(t,r)^{-1}|\partial_r Lv(t,r)| \lesssim \widetilde{N}_1(|v|^p).
\end{align}
Therefore, if we show that 
\begin{align}\label{1+2 cond Lv n even}
r^{m-1}\langle r\rangle \phi_\kappa(t,r)^{-1}| Lv(t,r)| \lesssim \widetilde{N}_1(|v|^p),
\end{align} then, we are done. We distinguish again two subcases.
If $r\leq 1$, then, by the estimate \eqref{2nd cond Lv n even} we find
\begin{align*}
r^{m-1}\langle r\rangle \phi_\kappa(t,r)^{-1}| Lv(t,r)| &\lesssim \langle r\rangle \langle t-r \rangle^{-(\kappa+\frac{1}{2})}\phi_\kappa(t,r)^{-1} \widetilde{N}_1^\nu(|v|^p)\\ & = \langle r\rangle \Big(\tfrac{ \langle t+r \rangle}{ \langle t-r \rangle}\Big)^{\frac{1}{2}} \widetilde{N}_1^\nu(|v|^p)\lesssim  \widetilde{N}_1^\nu(|v|^p), 
\end{align*} where in the last inequality we used the fact that $\tfrac{ \langle t+r \rangle}{ \langle t-r \rangle}$ is bounded in this case.
On the other hand, if $r\geq 1$,  since $r\approx\langle r\rangle$, \eqref{1st cond Lv n even} implies
\begin{align*}
r^{m-1}\langle r\rangle \phi_\kappa(t,r)^{-1}| Lv(t,r)| \lesssim   r^{-1}\langle r\rangle \widetilde{N}_1^\nu(|v|^p) \approx \widetilde{N}_1^\nu(|v|^p).
\end{align*}

Combining \eqref{3rd cond Lv der n even, rewritten} and \eqref{1+2 cond Lv n even}, we got \eqref{||Lv|| n even estimate}. This concludes the proof.
\end{proof}

In the next result we take a closer look to the relation between $\|\cdot\|_{X_\kappa}$ and $N_j^\nu (\cdot)$.

\begin{lemma} Let us consider $p,\kappa,q$ satisfying \eqref{p>p0(n+mu1) prel n even}, \eqref{upper bound for kappa n even}, \eqref{lower bound for kappa n even}, \eqref{q condition n even} and \eqref{p<p Fuj (mu1)}. 
Then, for any $v\in X_\kappa$
\begin{align*}
N_j^\nu(|v|^p)\lesssim \| v\|_{X_\kappa}^p \qquad \mbox{with} \quad \nu\doteq m-(m-1)p.
\end{align*} 

In particular, for any $v\in X_\kappa$
\begin{align}\label{Lv < v^p case n even}
\| Lv\|_{X_\kappa} \lesssim \| v\|_{X_\kappa}^p.
\end{align}
\end{lemma}

\begin{proof}
Let $v\in X_\kappa$. We start with $N_0^\nu(|v|^p)$. If $\tau\geq 0$ and $\lambda>0$, then, by using the definition of $\|\cdot\|_{X_\kappa}$, we obtain 
\begin{align*}
\lambda^{2m}|v(\tau,\lambda)|^p &\lambda^{-m-\nu}\langle\lambda\rangle^{q-\frac{p}{2}+\frac{3}{2}+\nu}\phi_\kappa(\tau,\lambda)^{-p} 
\lesssim \lambda^{m-(m-1)p-\nu}\langle\lambda\rangle^{q-\frac{3}{2}p+\frac{3}{2}+\nu}\|v\|_{X_\kappa}^p =\|v\|_{X_\kappa}^p .
\end{align*}
Let us remark that we used $\nu=m-(m-1)p$ in the last line, which is equivalent to  $\nu=-q+\frac{3}{2}(p-1)$. So, the supremum of the left-hand side, also known as $N_0^\nu(|v|^p)$, can be estimated by $\|v\|_{X_\kappa}^p$ with this choice of $\nu$.

Similarly, for $\tau\geq 0$ and $ \lambda>0$
\begin{align*}
\big|\partial_\lambda (\lambda^{2m}|v(\tau,\lambda)|^p)\big|  & \lesssim \lambda^{2m} |\partial_\lambda v(\tau,\lambda)| |v(\tau,\lambda)|^{p-1}+\lambda^{2m-1}  |v(\tau,\lambda)|^{p} \\ 
& \lesssim \lambda^{2m-1-(m-1)p}\langle \lambda\rangle^{-p+1} \phi_\kappa(\tau,\lambda)^p \|v\|_{X_\kappa}^p
\end{align*} implies 
\begin{align*}
\big|\partial_\lambda (\lambda^{2m}|v(\tau,\lambda)|^p)\big|\lambda^{-m-\nu+1}\langle\lambda\rangle^{q-\frac{p}{2}+\frac{1}{2}+\nu}\phi_\kappa(\tau,\lambda)^{-p} \lesssim 
\|v\|_{X_\kappa}^p .
\end{align*} 

Taking the supremum of the left-hand side in the previous inequality, we obtain the inequality $N_0^\nu(|v|^p)\lesssim \|v\|_{X_\kappa}^p  $, for $\nu$ as before.

Finally, we prove \eqref{Lv < v^p case n even}. It is sufficient to use \eqref{||Lv|| n even estimate}, provided that $$\nu=m-(m-1)p>-1.$$
 The previous condition is equivalent to require $p<\frac{m+1}{m-1}=\frac{n}{n-4}$ for $n\geq 6$ (in the case $n=4$ the condition $\nu>-1$ is always true, being $\nu=m$). 
However, the upper bound for $p$ in \eqref{p>p0(n+mu1) prel n even} is smaller than $\frac{n}{n-4}=p_{\Fuj}(\frac{n-4}{2})$.
 Therefore, $m-(m-1)p>-1$ is fulfilled under the assumptions of this lemma. The proof is completed.
\end{proof}

The next step is to prove the H\"older continuity of $L$ and the Lipschitz continuity of $L$ with respect to a different norm.
For this purpose we introduce an auxiliary norm on $X_\kappa$. For any $v\in X_\kappa$  we define 
\begin{align*}
\vertiii{v}_{X_\kappa}\doteq \sup_{t\geq 0\,, \, r>0} r^{m} |v(t,r)| \phi_\kappa (t,r)^{-1}.
\end{align*}
We note that $\vertiii{v}_{X_\kappa}\leq \|v\|_{X_\kappa}$ for $v\in X_\kappa$.

\begin{lemma} 
 Let us consider $p,\kappa,q$ satisfying \eqref{p>p0(n+mu1) prel n even}, \eqref{upper bound for kappa n even}, \eqref{lower bound for kappa n even}, \eqref{q condition n even} and \eqref{p<p Fuj (mu1)}. 
Then, for any $v,\bar{v}\in X_\kappa$
\begin{align}
N_0^{\nu_0}(|v|^p-|\bar{v}|^p)&\lesssim \vertiii{v-\bar{v}}_{X_\kappa}\big(\|v\|_{X_\kappa}^{p-1}+\|\bar{v}\|_{X_\kappa}^{p-1}\big),\label{est N0 nu0 Lv-Lvbar}\\
N_0^{\nu_1}(|v|^p-|\bar{v}|^p)&\lesssim \|v-\bar{v}\|_{X_\kappa}\big(\|v\|_{X_\kappa}^{p-1}+\|\bar{v}\|_{X_\kappa}^{p-1}\big),\label{est N0 nu1 Lv-Lvbar}\\
N_1^{\nu_2}(|v|^p-|\bar{v}|^p)&\lesssim \|v-\bar{v}\|_{X_\kappa}\big(\|v\|_{X_\kappa}^{p-1}+\|\bar{v}\|_{X_\kappa}^{p-1}\big)  + \vertiii{v-\bar{v}}_{X_\kappa}^{p-1}\big(\|v\|_{X_\kappa}+\|\bar{v}\|_{X_\kappa}\big),\label{est N1 nu2 Lv-Lvbar}
\end{align}  where $\nu_0\doteq (m-1)(1-p),\nu_1\doteq m-(m-1)p$ and $\nu_2\doteq m+1-(m-1)p$ and $N_j^\nu(|v|^p-|\bar{v}|^p)$ is defined analogously to \eqref{def Nj nu v^p} with $|v|^p-|\bar{v}|^p$ in place of $|v|^p$.

In particular, the following inequalities are satisfied for any $v,\bar{v}\in X_\kappa$:
\begin{align}\label{Lv -Lvbar lipschitz case n even}
\vertiii{ Lv -L\bar{v}}_{X_\kappa} &\lesssim \vertiii{v-\bar{v}}_{X_\kappa}\big(\|v\|_{X_\kappa}^{p-1}+\|\bar{v}\|_{X_\kappa}^{p-1}\big), \\
\| Lv -L\bar{v}\|_{X_\kappa} &\lesssim \|v-\bar{v}\|_{X_\kappa}\big(\|v\|_{X_\kappa}^{p-1}+\|\bar{v}\|_{X_\kappa}^{p-1}\big)+ \vertiii{v-\bar{v}}_{X_\kappa}^{p-1}\big(\|v\|_{X_\kappa}+\|\bar{v}\|_{X_\kappa}\big). 
\label{Lv -Lvbar holder case n even}
\end{align}
\end{lemma}

\begin{proof} Let $v,\bar{v}\in X_\kappa$.
For the sake of brevity, we use throughout the proof the notations $\widetilde{G}(\tau,\lambda)\doteq |v(\tau,\lambda)|^p-|\bar{v}(\tau,\lambda)|^p$ and 
\begin{align*}
M_1 &\doteq\|v-\bar{v}\|_{X_\kappa}\big(\|v\|_{X_\kappa}^{p-1}+\|\bar{v}\|_{X_\kappa}^{p-1}\big), \qquad 
& M_2 \doteq \vertiii{v-\bar{v}}_{X_\kappa}^{p-1}\big(\|v\|_{X_\kappa}+\|\bar{v}\|_{X_\kappa}\big), \\
M_3 &\doteq \vertiii{v-\bar{v}}_{X_\kappa}\big(\|v\|_{X_\kappa}^{p-1}+\|\bar{v}\|_{X_\kappa}^{p-1}\big).
\end{align*}

Using the definitions of $\|\cdot\|_{X_\kappa}$ and $\vertiii{\cdot}_{X_\kappa}$, we arrive at 
\begin{align}
|\lambda^{2m} \widetilde{G}(\tau,\lambda)|&\lesssim \lambda^{2m-1-(m-1)p}\langle\lambda\rangle^{-p+1}\phi_\kappa(\tau,\lambda)^{p} M_3,\label{G tilde M3}\\
|\lambda^{2m} \widetilde{G}(\tau,\lambda)|&\lesssim \lambda^{2m-(m-1)p}\langle\lambda\rangle^{-p}\phi_\kappa(\tau,\lambda)^{p} M_1, \label{G tilde M1}
\end{align}
for any $\tau\geq 0, \lambda>0$.

In a similar way, since the derivative of $|v|^p$ is a $(p-1)-$H\"{o}lder continuous function, then,
\begin{align}
\label{G tilde der M1+M2}  \big|\lambda^{2m} \partial_\lambda \widetilde{G}(\tau,\lambda) \big|  
&\lesssim \lambda^{2m} | v(\tau,\lambda)|^{p-1}|\partial_\lambda v(\tau,\lambda)-\partial_\lambda \bar{v}(\tau,\lambda)|\\ & \quad +\lambda^{2m}|v(\tau,\lambda)-\bar{v}(\tau,\lambda)|^{p-1}|\partial_\lambda \bar{v}(\tau,\lambda)| \notag \\
&\lesssim \lambda^{2m-(m-1)p-1}\langle\lambda\rangle^{-p+1}\phi_{\kappa}(\tau,\lambda)^p M_1+ \lambda^{2m-mp}\phi_{\kappa}(\tau,\lambda)^p M_2.  \notag
\end{align} 

Let us derive \eqref{est N0 nu0 Lv-Lvbar}. Using \eqref{G tilde M3}, we get immediately
\begin{align*}
\lambda^{2m}|\widetilde{G}(\tau,\lambda)| \lambda^{-m-\nu_0} \langle\lambda\rangle^{m(p-1)+\nu_0}\phi_\kappa(\tau,\lambda)^{-p} &\lesssim 
M_3.
\end{align*} Here we used the equality $m(p-1)=q-\frac{p}{2}+\frac{3}{2}$. Thus, taking the supremum of the left-hand side for $\tau\geq 0,\lambda>0$, we get \eqref{est N0 nu0 Lv-Lvbar}.

Analogously, one can prove \eqref{est N0 nu1 Lv-Lvbar}, by using \eqref{G tilde M1}. Let us derive now \eqref{est N1 nu2 Lv-Lvbar}. By \eqref{G tilde M1} and \eqref{G tilde der M1+M2}, it follows:
\begin{align*}
\big|\partial_\lambda(\lambda^{2m} \partial_\lambda \widetilde{G}(\tau,\lambda)) \big|   \lambda^{-m-\nu_2+1} \langle\lambda\rangle^{m(p-1)+\nu_2-1}\phi_\kappa(\tau,\lambda)^{-p} 
& \lesssim \lambda^{p-1}\langle\lambda\rangle^{-(p-1)} M_1 +  M_2 \\ &\leq M_1+M_2,
\end{align*} which implies \eqref{est N1 nu2 Lv-Lvbar}.

It remains to prove \eqref{Lv -Lvbar lipschitz case n even} and \eqref{Lv -Lvbar holder case n even}. 
First of all, let us remark that $\nu_2>-1$ if and only if $p<\frac{m+2}{m}=\frac{n+2}{n-2}=p_{\Fuj}\big(\frac{n-2}{2}\big)$. Nevertheless, the upper bound for $p$ in \eqref{p>p0(n+mu1) prel n even} is smaller then $p_{\Fuj}\big(\frac{n-2}{2}\big)$, therefore, the condition $\nu_2>-1$ is always fulfilled under the assumptions of this lemma.
Secondly, $\nu_0=\nu_1-1$ and $\nu_1>\nu_2$. Hence, $\nu_2>-1$ implies $\nu_1>-1$ and $\nu_0>-2$.

According to Remark \ref{rmk lispscitz conditions}, analogously to \eqref{1st cond Lv n even}, it is possible to show that
\begin{align*}
|Lv(t,r)-L\bar{v}(t,r)|\lesssim N_0^\nu(|v|^p-|\bar{v}|^p) \,r^{-m} \phi_\kappa(t,r) \qquad \mbox{for} \quad \nu>-2.
\end{align*}
In particular, the previous condition for $\nu_0>-2$ implies, together with \eqref{est N0 nu0 Lv-Lvbar}, the Lipschitz condition \eqref{Lv -Lvbar lipschitz case n even}.

Similarly, replacing the source term $|v|^p$ with the difference $|v|^p-|\bar{v}|^p$, analogously to \eqref{||Lv|| n even estimate} we obtain
\begin{align*}
\| Lv-L\bar{v}\|_{X_\kappa}\lesssim N_0^\nu(|v|^p-|\bar{v}|^p)+N_1^\nu(|v|^p-|\bar{v}|^p) \qquad \mbox{for} \quad \nu>-1.
\end{align*} 
Also, because of $\nu_2>-1$, employing \eqref{est N0 nu1 Lv-Lvbar} and \eqref{est N1 nu2 Lv-Lvbar} and using the the fact $N_j^\nu$ is increasing with respect to $\nu$, we have
\begin{align*}
\| Lv-L\bar{v}\|_{X_\kappa}& \lesssim N_0^{\nu_2}(|v|^p-|\bar{v}|^p)+N_1^{\nu_2}(|v|^p-|\bar{v}|^p) \\& \lesssim N_0^{\nu_1}(|v|^p-|\bar{v}|^p)+N_1^{\nu_2}(|v|^p-|\bar{v}|^p) \lesssim M_1+M_2,
\end{align*} since $\nu_1>\nu_2$. This condition is exactly \eqref{Lv -Lvbar holder case n even}, so, the proof is over.
\end{proof}

\begin{proof}[Proof of Theorem \ref{thm GER n even semi}] Let $\kappa_1$ and $\kappa_2$ be defined as in Remark \ref{rmk after thm GER n even}. Let us fix a $\kappa$ in $(\kappa_1,\bar{\kappa}]$.
Considering the transformed Cauchy problem \eqref{semi rad n>=5}, according to our setting it is enough to prove that the operator 
\begin{align*}
F v= v^0+Lv \qquad \mbox{for any} \quad v \in X_\kappa
\end{align*} admits a uniquely determined fixed point on a ball in $X_\kappa$ around $0$ with sufficiently small radius, where $v^0$ is defined by \eqref{def v0 even case}. Thanks to Proposition \ref{thm lin probl n>=4 distr}, we get $\| v^0\|_{X_\kappa}\lesssim \varepsilon$. 

Since $L$ satisfies \eqref{Lv < v^p case n even}, \eqref{Lv -Lvbar lipschitz case n even} and \eqref{Lv -Lvbar holder case n even}, following the approach used in the proof of Proposition 5.4  in \cite{KuKu95}, we find that $F$ has a uniquely determined fixed point, which is the desired solution, provided that  $\varepsilon <\varepsilon_0$ for a suitably small $\varepsilon_0>0$.
\end{proof}

\section{Concluding remarks and open problems} \label{Section final rmk}

Finally, we point out the main consequences of Theorem \ref{thm GER n even semi}.
When $\mu=2$ and $\nu=0$, then, \eqref{semilinear scale inv damping and mass} coincides with the semilinear model studied in \cite{DabbLucRei15}. Therefore, combining the result proved in Theorem \ref{thm GER n even semi} with the blow-up result \cite[Theorem 1]{DabbLucRei15} and the global (in time) existence results \cite[Theorem 2]{D13}, \cite[Theorems 2 and 3]{DabbLucRei15} and \cite[Theorem 2.1]{DabbLuc15}, we find that $p_2(n)$, defined as in \eqref{def p2}, is the critical exponent for the semilinear model \eqref{semilinear scale inv damping} when $\mu=2$ as it is conjectured in \cite{DabbLucRei15}.

Similarly, combining Theorem \ref{thm GER n even semi} with the blow-up result \cite[Theorem 2.6]{NunPalRei16} and the global existence results from \cite{Pal18}, we have that $p_{\crit}(n,\mu)$ is critical exponent for the model \eqref{semilinear scale inv damping and mass} assuming \eqref{delta=1 condition} for $n=1$ and for $n\geq 3$ in the radial symmetric case with $\mu\leq M(n)$. The two-dimensional case is still open, even though from the necessity part we expect $p_{\Fuj}(1+\tfrac{\mu}{2})$ to be critical.

In this paper and in \cite{Pal18}, we restrict our consideration to the case in which $\mu$ and $\nu$ satisfy \eqref{delta=1 condition}. However, recently in \cite{PalRei17FvS} a blow-up result has been shown for $\delta\in(0,1]$ for $$1<p\leq \max\big\{p_{\Fuj}\big(n+\tfrac{\mu-1}{2}-\tfrac{\sqrt{\delta}}{2}\big),p_0(n+\mu)\big\},$$ excluding the case $p=p_0(n+\mu)$ for $n=1$. Consequently, a challenging open problem is to study the necessary part also in the case $\delta\in (0,1)$, showing that the previous upper bound is actually critical.

\begin{small}
\paragraph{Acknowledgement} The PhD study of the author is supported by S\"{a}chsiches Landesgraduiertenstipendium. The author is member of the Gruppo Nazionale per L'Analisi Matematica, la Probabilit\`{a} e le loro Applicazioni (GNAMPA) of the Instituto Nazionale di Alta Matematica (INdAM). Moreover, the author thanks his supervisor Michael Reissig (TU Freiberg) for strongly recommending the study of the topic of this work.
\end{small}


\end{document}